\documentclass[11pt, a4paper]{amsart} 

\title[Cyclic operad formality for genus zero moduli spaces]{Cyclic
  operad formality for compactified moduli spaces of genus zero
  surfaces}

\date{22 July, 2010}

\author{Jeffrey Giansiracusa}
\address{
Department of Mathematics, Swansea University \\ 
Singleton Park\\
Swansea \\
SA2 8PP\\ 
United Kingdom}
\email{j.h.giansiracusa@gmail.com}

\author{Paolo Salvatore} 
\address{
Dipartimento di Matematica \\
Universita' di Roma ``Tor Vergata'' \\
Via della Ricerca Scientifica \\
00133 Roma  \\
ITALY}
\email{salvator@mat.uniroma2.it}


\usepackage{amssymb}
\usepackage{pst-all} 

\usepackage{mathrsfs}  

\usepackage[cmtip,arrow]{xy}
\usepackage{pb-diagram,pb-xy}
\usepackage[vmargin=3cm, hmargin=3.2cm,dvips]{geometry}
\numberwithin{equation}{subsection}
\newtheorem{theorem}{Theorem}[section] 
\newtheorem{lemma}[theorem]{Lemma} 
\newtheorem{proposition}[theorem]{Proposition}

\newtheorem{definition}[theorem]{Definition}

\newtheorem{thmA}{Theorem}

\theoremstyle{remark} 

\newcommand{\id}{\mathrm{id}}

\newcommand{\C}{\mathbb{C}}
\newcommand{\R}{\mathbb{R}}
\newcommand{\N}{\mathbb{N}}

\newcommand{\Q}{\mathbb{Q}}

\newcommand{\I}{\mathscr{I}}
\newcommand{\M}{f\underline{\mathcal{M}}}
\newcommand{\uM}{\underline{\mathcal{M}}}
\def\co{\colon\thinspace} 
\newcommand{\coloneqq}{\mathbin{\mathpalette{\vcenter{\hbox{$:$}}}=}}

\newcommand{\proj}{\mathcal{PG}}
\newcommand{\aff}{\mathcal{AG}}
\newcommand{\kont}{\mathcal{KG}}
\newcommand{\xproj}{\widetilde{\mathcal{PG}}}
\newcommand{\xaff}{\widetilde{\mathcal{AG}}}

\begin{document}
\begin{abstract}
  The framed little 2-discs operad is homotopy equivalent to the
  Kimura-Stasheff-Voronov cyclic operad of moduli spaces of genus zero
  stable curves with tangent rays at the marked points and nodes. We
  show that this cyclic operad is formal, meaning that its chains and
  its homology (the Batalin-Vilkovisky operad) are quasi-isomorphic
  cyclic operads.  To prove this we introduce a new complex of graphs
  in which the differential is a combination of edge deletion and
  contraction, and we show that this complex resolves BV as a cyclic
  operad.
\end{abstract}
\maketitle

\section{Introduction}

We begin by recalling two closely related operads.  The first is the
\emph{little 2-discs} operad of Boardman and Vogt, denoted $D_2$.  The
arity $n$ space, $D_2(n)$, is the space of embeddings of the disjoint union of
$n$ discs into a standard disc, where each disc is embedded by a map
which is a translation composed with a dilation.  At the level of
spaces, group complete algebras over this operad are 2-fold loop
spaces, and at the level of homology an algebra over $H_*(D_2)$ is
precisely a Gerstenhaber algebra.

The second operad is a variant of the little 2-discs; it is the
\emph{framed little 2-discs} operad, denoted $fD_2$, first introduced
by Getzler \cite{Getzler}.  Here each little disc is allowed to be
embedded by a composition of a \emph{rotation} in addition to a
dilation and translation, so in particular, $fD_2(n)=D_2(n) \times
(S^1)^n$.  Getzler observed that algebras over the homology operad
$H_*(fD_2)$ are precisely Batalin-Vilkovisky algebras, and at the
space level the second author and Wahl \cite{SW} proved that a group complete
algebra over $fD_2$ is a 2-fold loop space on a based space with a
circle action.

Tamarkin \cite{Tamarkin} first proved that the operad $D_2$ is
\emph{formal} over $\Q$, meaning that the singular chain operad
$C^{sing}_*(D_2;\Q)$ is quasi-isomorphic to the homology operad
$H_*(D_2;\Q)$.  His proof is algebraic in nature, using braid groups
and Drinfeld associators.  Kontsevich \cite{Ko1, Ko} sketched a
different proof of this formality theorem over $\R$ which generalizes
to show formality of the little $k$-discs for all $k\geq 2$.  Roughly
speaking, Kontsevich works dually, constructing a cooperad of cochain
complexes of graphs which maps to both the cochains and the cohomology
of $D_2$ by quasi-isomorphisms.  The full details of Kontsevich's
proof have been explained nicely by Lambrechts and Volic
\cite{LambVol}.  The Kontsevich proof method has the advantage over
Tamarkin's proof of simultaneously proving the formality in the
Sullivan DGA sense of the individual spaces of the operad.  Recently
Severa and Willwacher \cite{SevWill} have shown that the
quasi-isomorphisms of Kontsevich and Tamarkin are homotopic for a
particular choice of Drinfeld associator.  Tamarkin's proof of
formality for $D_2$ has been adapted to the operad $fD_2$ by Severa
\cite{Severa}, and Kontsevich's proof has been adapted to $fD_2$ by
the present authors in \cite{GS1}.  Both adaptations are relatively
straightforward.

Operad formality theorems have significant applications.  The
formality of $D_2$ plays an important role in Tamarkin's proof
\cite{Tamarkin-deformation} of Kontsevich's celebrated deformation
quantization theorem, and the formality of $D_k$ for $k \geq 4$ is a
primary ingredient in the computation of the homology of spaces of
long knots in high dimensions \cite{LambTurch}.  Formality of the
operad $fD_2$ is used in \cite{Vallette} to construct chain level
homotopy BV-algebra structures for topological conformal field
theories and for the 2-fold loop space on a space with a circle
action, as well as to give a particular solution to the cyclic Deligne
conjecture.

Unlike $D_k$, the operad $fD_k$ has the additional structure of being
homotopy equivalent to a \emph{cyclic operad} (see \cite{Budney} for
an explicit cyclic model); being a cyclic operad means roughly that
the roles of inputs and outputs can be interchanged.  The present
paper was motivated by the question of whether the formality of the
operad can be made compatible with the cyclic structure.

We answer this question in the affirmative when $k=2$.  In section
\ref{cyclic-model} we will recall a cyclic model $\M$ for $fD_2$
constructed in terms of compactified moduli spaces of genus zero
curves with marked points. The elements of $\M$ are genus zero stable
curves decorated with real tangent rays at the marked points and
nodes. The operad composition is obtained by gluing stable curves at
marked points and tensoring the corresponding rays.  This model was
first introduced by Kimura-Stasheff-Voronov \cite{KSV}, and it is
closely related to the modular operad of Deligne-Mumford compactified
moduli spaces of curves whose formality is proved in \cite{Spaniards}.
Moreover $\M$ is homotopy equivalent, as a cyclic operad, to the
operad of genus zero Riemann surfaces with boundary, where the operad
composition is defined by gluing along boundary components.

The main result of this paper is:
\begin{thmA}\label{main}
  The cyclic operad $C^{sing}_*(\M;\R)$ of singular chains on $\M$
  with real coefficients is quasi-isomorphic as a cyclic operad to its
  homology, $H_*(\M;\R)=BV$, which is the Batalin-Vilkovisky operad.
\end{thmA} 

The new and important point of this theorem is that the chain of
quasi-isomorphisms is compatible with the \emph{cyclic} operad
structure.  A second point, coming out of the proof, is that this
operad formality simultaneously realises the formality of the
individual spaces of the operad in the Sullivan commutative DGA sense,
as was the case for the Kontsevich proof in the unframed case.  This
cyclic formality has a consequence in computing the cohomology of
diffeomorphism groups of 3-dimensional handlebodies, as discussed
further below.

It is not immediately clear that the previous operad formality proofs
given in \cite{GS1} and \cite{Severa} are compatible with the cyclic
structure, nor is it clear that they are compatible with the formality
of the individual spaces of the operad.  However, since the completion
of this work Severa has indicated how to show that his formality proof
is in fact compatible with the cyclic structure.  As for the other
proof, the Kontsevich graph complex, as modified for the framed
2-discs operad in \cite{GS1}, in arity $n$ has an action of the
enlarged symmetric group $\Sigma_{n+1}$ (this action would be part of
the cyclic operad structure) by linear automorphisms, but it is not
clear if this action respects the differential.  It seems likely that
the action is in fact compatible, but proving this appears to be a
very difficult combinatorial problem.  

Instead, here we take a more conceptual approach and build a new
manifestly cyclic operad of ``projective graph complexes'' as a
replacement for Kontsevich's graph complexes.  The differential is now
a combination of both edge deletions and edge contractions.  The
projective graph complexes do form a cyclic cooperad.  Our proof of
Theorem \ref{main} then follows in outline the method used by
Kontsevich in the unframed case.  The projective graph complexes map
onto the cohomology of $\M$, and they also map to the (semi-algebraic)
forms on $\M$ by a variation of Kontsevich's map defined by certain
configuration space integrals.  We prove that both of these maps are
quasi-isomorphisms and are compatible with the cyclic cooperad
structures.  The theorem then follows by dualising.

For $k>2$ the question of formality of $fD_k$ as a cyclic operad, or
even just as an ordinary operad, remains open.  In future work we
hope to address this.

\subsection{Application to 3-dimensional handlebodies}
In the forthcoming paper \cite{handlebodies}, the first author proves
that the modular operad generated by the cyclic operad $\M$ is
homotopy equivalent to the modular operad made from $B\mathrm{Diff}$s of
handlebodies.  This leads to a Bousfield-Kan spectral sequence computing the
cohomology of $B\mathrm{Diff}(H_g)$ for $H_g$ a handlebody of genus
$g$.  Theorem \ref{main} implies that this spectral sequence
degenerates at the $E^2$ page.  (Note that formality as an operad rather than a
cyclic operad is not enough to imply degeneration of the spectral sequence.)


\section{The framed 2-discs as a cyclic operad}\label{cyclic-model}

The framed little 2-discs operad $fD_2$ is homotopy equivalent to a
cyclic operad, (although it is not itself cyclic on the nose).  A
convenient cyclic model was introduced in \cite[\S 3.4]{KSV}.  We
recall that model and discuss some of its properties.

\subsection{Cyclic operads and cooperads}\label{cyclic-operad-definitions}
Recall that a \emph{cyclic operad} is an extension of the concept of 
operad in which inputs and outputs are on equal footing; this can be
formalised by asking that the $\Sigma_n$ action on the arity $n$ space
extends to an action of $\Sigma_{n+1}$ in a way compatible with the
operad composition maps.  The notion of a cyclic operad was first
introduced by Getzler and Kapranov \cite{GetzKapr}.

In more detail, a cyclic operad in a symmetric monoidal category
$\mathcal{C}$ is defined by the following data: a functor $P$ from the
category of non-empty finite sets and bijections to $\mathcal{C}$, and
for each pair of finite sets $I,J$ with elements $i\in I, \, j \in J$
a composition morphism,
\[
_i\circ_j \colon P(I) \otimes P(J) \to P(I\sqcup J \smallsetminus \{i,j\}),
\]
natural in $I$ and $J$ and satisfying the following axioms.
\begin{itemize}
\item \textbf{(Associativity)}
Given finite sets $I,J,K$, and elements $i \in I, \, j_1,j_2 \in J, \,
k \in K$, the following diagram commutes:
\[
\begin{diagram}
\node{P(I) \otimes P(J) \otimes P(K)} 
\arrow{e,t}{\id \otimes ({_{j_2}\circ_k})}
\arrow{s,r}{({_i\circ_{j_1}}) \otimes \id} 
\node{P(I) \otimes P(J\sqcup K \smallsetminus \{j_2,k\})} \arrow{s,r}{_i\circ_{j_1}}\\
\node{P(I\sqcup J \smallsetminus \{i,j_1\})\otimes P(K)} 
\arrow{r,t}{{_{j_2}\circ_k}} 
\node{P(I\sqcup J\sqcup K \smallsetminus \{i,j_1,j_2,k\}).}
\end{diagram}
\]

\item \textbf{(Commutativity)} Given finite sets $I,J$ and elements
$i \in I$ and $ j \in J$, the following diagram commutes:
\[
\begin{diagram}
\node{P(I)\otimes P(J)} \arrow{e,t}{_i\circ_j} \arrow{s,l}{\mathrm{swap}}
\node{P(I\sqcup J \smallsetminus \{i,j\})} \arrow{s,r}{P(\tau)}\\
\node{P(J)\otimes P(I)} \arrow{e,t}{_j\circ_i}
\node{P(J\sqcup I \smallsetminus \{j,i\}),}
\end{diagram}
\]
where $\tau$ is the canonical bijection $I\sqcup J \smallsetminus
\{i,j\} \cong J\sqcup I \smallsetminus \{j,i\}$.

\item \textbf{(Unit axiom)} For each set $A=\{a,b\}$ of cardinality 2
  there is a morphism $u_A\colon 1_{\mathcal{C}} \to P(A)$, where
  $1_{\mathcal{C}}$ is the monoidal unit, that is natural in $A$ and
  such that for any finite set $I$ and an element $i \in I$ the
  following diagram commutes:
\[
\begin{diagram}
\node{1_{\mathcal{C}}\otimes P(I)} \arrow{e,t}{u_A \otimes \id}
 \arrow{se,b}{P(\phi)} \node{P(A)\otimes P(I)} \arrow{s,r}{_a\circ_i}
 \\
\node[2]{P(A\sqcup I \smallsetminus \{a,i\}),}
\end{diagram}
\]
where $\phi\colon I \cong A\sqcup I \smallsetminus \{a,i\}$ is the
canonical bijection sending $i$ to $b$.
\end{itemize}
  
Dually one defines a \emph{cyclic cooperad} by reversing the arrows in
the definition.  Concretely, a cyclic cooperad is specified by the
data of a functor $P$ from finite sets and bijections to
$\mathcal{C}$, and a co-composition map
\[
_K\bullet_L \colon P(V) \to P(K\sqcup \{x\}) \otimes P(L\sqcup \{y\}),
\]
(where $x$ and $y$ are auxiliary points)
for each partition $V=K\sqcup L$.  These co-composition maps must
satisfy dual versions of the naturality, symmetry, and associativity
conditions.  

A cyclic (co)operad in the category of chain complexes is called a
{\em differential graded cyclic (co)operad}.

\subsection{The compactified genus zero moduli space model} 
We now recall the moduli space operad introduced by Kimura, Stasheff
and Voronov in \cite[\S 3.4]{KSV}.  In their notation, the operad is
denoted $\underline{\mathcal{N}}$, but we prefer $\M$ as being more
consistent with other notations.  This operad is manifestly a cyclic
operad, and we will show in Proposition \ref{operad-equivalence} that
its underlying operad is homotopy equivalent to the framed little
2-discs operad.

Given a non-empty finite set $V$, let $\mathcal{M}_V$ denote the moduli
space of genus zero Riemann surfaces with marked points labelled by
$V$.  Explicitly,
\[
\mathcal{M}_V \coloneqq \mathrm{Inj}(V,\mathbb{CP}^1)/PGL_2(\C),
\]
where $\mathrm{Inj}(V,\mathbb{CP}^1)$ is the space of injective maps
from $V$ to the complex projective line, and the automorphism group $PGL_2(\C)$ acts on the
target in the standard way.  We will write $\mathcal{M}_n$ when $V$ is
the set $\{0, \ldots, n-1\}$. 

There is an identification between $\mathcal{M}_n$ and the complement
in $(\mathbb{CP}^1)^{n-3}$ of the union of the hyperplanes of the form
\begin{equation}\label{hyp} 
z_i=z_j, \,z_i=0, \, z_i=1, \, z_i=\infty \, .
\end{equation}

If $n > 3$ then the moduli space $\mathcal{M}_n$
is noncompact and has a Deligne-Mumford-Knudsen compactification
\[
\mathcal{M}_n \subset \overline{\mathcal{M}}_n
\]
which can be constructed blowing up in a suitable order all proper
intersections of hyperplanes of the form \eqref{hyp} in
$(\mathbb{CP}^1)^{n-3}$ \cite{keel}.  The points of
$\overline{\mathcal{M}}_n$ correspond to stable nodal Riemann
surfaces, (\emph{nodal} means that the surfaces are allowed to have
double-point singularities, and \emph{stable} in this case means that
no irreducible component is a sphere with strictly fewer than 3 points
that are either nodes or marked points).  For $i =0,\cdots (n-1)$ there is a
complex line bundle $\mathcal{L}_i$ over $\overline{\mathcal{M}}_n$
whose fibre over a given Riemann surface $S$ is the tangent space at
the $i^{th}$ marked point.

The boundary, $\overline{\mathcal{M}}_V \smallsetminus \mathcal{M}_V$,
is a normal crossing divisor.  It has an irreducible component for
each unordered partition $V=K \sqcup L$ into two subsets of cardinality at least
2, and the corresponding component is canonically isomorphic to
$\overline{\mathcal{M}}_{K \sqcup \{x\}} \times
\overline{\mathcal{M}}_{L \sqcup \{y\}}$. The inclusions
\[
\overline{\mathcal{M}}_{K \sqcup \{x\}} \times
\overline{\mathcal{M}}_{L \sqcup \{y\}} \to \overline{\mathcal{M}}_V
\]
define a cyclic operad structure on the Deligne-Mumford-Knudsen moduli
spaces.

The space $\uM_V$ is defined as the real oriented blowup of the
boundary locus of $\overline{\mathcal{M}}_V$.  We pull the complex
line bundles $\mathcal{L}_x$ back from $\overline{\mathcal{M}}_V$ (and
denote them by the same symbol).  The space $\uM_V$ can be interpreted
as the moduli space of genus zero stable nodal Riemann surfaces with
points marked by $V$ and equipped with the additional data of a ray in
the tensor product of the two tangent spaces at either side of each
node.  It can be constructed as the iterated \emph{real oriented}
blowup of $(\mathbb{CP}^1)^{n-3}$ along {\em all} intersections of
hyperplanes \eqref{hyp}, by blowing up first the proper intersections
in the same order as for the Deligne-Mumford space, and finally the
hyperplanes themselves.  The space $\overline{\mathcal{M}}_V$ is a
compact smooth manifold with \emph{faces}; i.e. it is a manifold with
corners, stratified with the interior as the unique codimension 0
stratum and all of the positive codimension strata lying on the
boundary, each of which is a transversal intersection of faces
(codimension 1 strata).  Faces correspond to unordered partitions of
$V$ into two subsets $K \sqcup L$ of cardinality at least 2; given
such a partition, the corresponding face is canonically isomorphic to
the unit circle bundle of the exterior tensor product $\mathcal{L}_x
\otimes \mathcal{L}_y$ over $\uM_{K \sqcup \{x\}} \times \uM_{L \sqcup
  \{y\}}.$

Let $f\mathcal{M}_V$ denote the (coarse) moduli space of genus zero
smooth surfaces equipped with points marked by $V$ with real tangent
rays.  Explicitly,
\[
f\mathcal{M}_V = \{ g\colon V \to S(T\mathbb{CP}^1) \quad | \quad \pi \circ g \in
\mathrm{Inj}(V,\mathbb{CP}^1)\}/PGL_2(\C),
\]
where $\pi\colon S(T\mathbb{CP}^1) \to \mathbb{CP}^1$ is the circle
bundle associated with the tangent bundle of $\mathbb{CP}^1$.  Observe
that there are diffeomorphisms
\begin{align*}
f\mathcal{M}_1 & \cong *,  \\
f\mathcal{M}_2 & \cong S^1 \mbox{ (canonical), and} \\
f\mathcal{M}_3 & \cong (S^1)^3 \mbox{ (non-canonical).}
\end{align*}
In particular, $f\mathcal{M}_n$ is compact for $n \leq 3$.

When $n > 3$, $f\mathcal{M}_n$ is noncompact.  Forgetting the
tangent rays at the marked points gives a trivial
principal $(S^1)^n$-bundle $f\mathcal{M}_n \to \mathcal{M}_n$ which
extends uniquely to a trivial torus bundle
over $\underline{\mathcal{M}}_n$.  We call its total space 
$\M_n \cong (S^1)^n \times \underline{\mathcal{M}}_n$.
Note that the trivialisation of the bundle is not canonical.
More explicitly,
\[
\M_n \coloneqq S(\mathcal{L}_0) \times_{\uM_n} \cdots \times_{\uM_n} S(\mathcal{L}_{n-1}) \, .
\]
One sees that $\M_n$ is a compactification of $f\mathcal{M}_n$
obtained by adding boundary and corners, and hence it is homotopy
equivalent to its interior.  When $n \leq 3$ we simply define $\M_n
\coloneqq f\mathcal{M}_n$.

The functor $V \mapsto \M_V$ forms a cyclic operad.  Composition with
$\M_2 \cong S^1$ acts by rotating tangent rays, and composition with
$\M_1\cong \{*\}$ removes a marked point.  All other composition maps
are defined by gluing marked points together to form a node and
tensoring their tangent rays to produce the decoration at the node; in
more detail, if $K$ and $L$ both have cardinality at least 2 then the
composition map
\[
_x\circ_y\colon \M_{K \sqcup \{x\}} \times \M_{L \sqcup \{y\}} \to \M_V
\]
is a trivial circle bundle over the face associated with the
partition $V= K\sqcup L$ and it is induced by the multiplication map
$S(\mathcal{L}_x) \times S(\mathcal{L}_y) \to S(\mathcal{L}_x\otimes
\mathcal{L}_y)$.

Recall that a homotopy equivalence of topological operads is a
morphism of operads (a natural transformation commuting with the
composition and unit maps) that is level-wise an equivariant homotopy
equivalence.
\begin{proposition}\label{operad-equivalence}
There is homotopy equivalences of operads
\[
fD_2 \simeq \M \, .
\]
There is also a homotopy equivalence of spaces, $D_2(n) \simeq
\M_{n+1} / (S^1)^n$, where the torus acts by rotating all but the
zeroth tangent direction.
\end{proposition}
\begin{proof}
  In \cite{Barcelona} the second author constructed a zigzag of
  homotopy equivalences between the operad $D_2$ and the
  \emph{Fulton-MacPherson operad}, $FM_2$, introduced by Getzler and
  Jones in \cite{GJ}.  The space $FM_2(n)$ is the real oriented
  Fulton-MacPherson compactification of the space of configurations of
  $n$ points in $\R^2$ modulo dilation and translation.  It can be
  constructed by iterated real oriented blowup from the
  $(2n-3)$-sphere $S^{2n-3} \cong (\C^n \smallsetminus
  \Delta(\C))/Aff_2$, where $\Delta(\C)$ is the thin diagonal defined
  by $z_1 = \dots =z_n$ and $\mathit{Aff}_2 \coloneqq \R_+ \ltimes \C$
  is the affine group generated by positive dilations and
  translations. The blowups are performed along all diagonals. The
  space $FM_2(n)$ is a manifold with faces; a face corresponds to an
  unordered partition of $\{0,1,\dots,n\}=K \sqcup L$ into two subsets
  of cardinality at least 2, and it is canonically diffeomorphic to
  $FM_2(|K|) \times FM_2(|L|)$. The inclusion of a face with
  $L=\{i,\dots,i+l-1\}$ represents a (non-cyclic) $\circ_i$ operad
  composition.

  From \cite{SW}, $D_2$ is an operad in the category of spaces with
  $SO(2)$ action and so one can form the semidirect product operad
  $D_2 \ltimes SO(2)$ which is isomorphic to $fD_2$.  Similarly,
  $FM_2$ is an operad in $SO(2)$-spaces, and one can form the
  \emph{framed Fulton-MacPherson operad} $fFM_2 \coloneqq FM_2 \ltimes
  SO(2)$.  The homotopy equivalence between $D_2$ and $FM_2$ is
  realised by a zigzag of $SO(2)$-equivariant maps and thus induces a
  homotopy equivalence of operads,
  \[
  fD_2 \simeq fFM_2.
  \]

  It turns out that $fFM_2$ and $\M$ are isomorphic as operads.  This
  is immediate in arity $n<3$ In higher arity, note that blowing up
  the sphere $(\C^n \smallsetminus \Delta(\C))/Aff_2$ at the diagonal
  $z_{n-1}=z_n$ is the same as blowing up $S^1 \times
  (\mathbb{CP}^1)^{n-3}$ at the hypersurfaces $z_i=\infty, \, 1 \leq i
  \leq n-3$.  Blowing up all other diagonals identifies $FM_2(n) \cong
  S^1 \times \uM_{n+1}$, and thus we have diffeomorphisms $fFM_2(n)
  \cong \M_{n+1}$ compatible with the operad composition, by
  definition.
\end{proof}

\subsection{Semi-algebraic structure and canonical
projections}\label{canonical-projections}
For the proof of Theorem \ref{main}, we will need to work in the category
of semi-algebraic spaces.  The theory of semi-algebraic spaces and its
corresponding de Rham theory has been developed by Hardt, Lambrechts,
Turchin and Volic in \cite{HLTV}.

\begin{proposition}
The functor $V \mapsto \M_V$ forms a cyclic operad in the
category of semi-algebraic manifolds.
\end{proposition}
\begin{proof}
  Sinha constructs in \cite{Sinha} a
  $\Sigma_n$-equivariant semi-algebraic embedding $j_n\colon FM_2(n)
  \hookrightarrow (S^1)^{\binom{n}{2}} \times
  [0,\infty]^{\binom{n}{3}}$ by evaluating the direction between each
  pair of points and ratio of distances for each triple of points.
  Thus the symmetric group acts semi-algebraically. Moreover,
  Lambrechts and Volic check in \cite[Prop 5.5]{LambVol} that the
  operad composition $\circ_i \colon FM_2(n) \times FM_2(m) \to
  FM_2(m+n-1)$ is the restriction of a semi-algebraic map.  Passing to
  the semi-direct product with $SO(2)$, it follows that $fFM_2$ is an
  operad in semi-algebraic manifolds.  Finally, we must verify that
  the action of $\Sigma_n$ on $\M_n \cong fFM_2(n-1)$ is
  semi-algebraic.  For this it is convenient instead to
  $\Sigma_n$-equivariantly embed $\M_n \hookrightarrow
  (S^1)^{\binom{n}{2}} \times (\mathbb{CP}^1)^{\binom{n}{4}}$ by
  evaluating angles between rays for pair of points and the cross
  ratio for 4-tuples of points.  This shows that $\Sigma_{n}$ acts
  semi-algebraically on $\M_n$.
\end{proof}

Forgetting subsets of the marked points defines a collection of
important projections between the various moduli spaces.  Given a
subset $A\subset V$, let
\[
\pi^A \colon \M_V \to \M_{V\smallsetminus A}
\]
be the map induced by forgetting all marked points labelled by $A$.
Note that $\pi^A$ is the iterated operad composition with the point of
$\M_1$ for each element of $A$.  

\begin{proposition}
The map $\pi^{A} \colon \M_{V} \to \M_{V\smallsetminus A}$, forgetting the points in
$A$, is a semi-algebraic fibre bundle.
\end{proposition} 
\begin{proof}
  The projection $\pi^{A}$ is equivalent to the composition
  $(S^1)^{|V|-1} \times FM_2(|V|-1) \to (S^1)^{|V|-1} \times
  FM_2(|V\smallsetminus A|-1) \to (S^1)^{|V \smallsetminus A|-1} \times
  FM_2(|V \smallsetminus A|-1) $.  The first map is a semi-algebraic
  bundle by \cite[Appendix A]{LambVol}, the second map is a trivial
  bundle, and the composite of bundle projections is a bundle
  projection by \cite[Prop. 8.5]{HLTV}.
\end{proof}

The fibres of the bundle $\pi^{A}$ are manifolds with faces, so
the fibrewise boundary is stratified into pieces of codimension 1 (in
the total space) and greater.  We shall refer to the codimension 1
strata as the \emph{fibrewise boundary components}.
\begin{lemma}\label{fibrewise-boundary}
  For $|V\smallsetminus A| \geq 3$, the fibrewise boundary components
  of the bundle
  \[\pi^{A} \colon \M_{V} \to \M_{V\smallsetminus A}\] correspond bijectively with the
  set of unordered partitions of $V$ into two subsets $U_1,
  U_2$ of cardinality at least 2 such that $|U_1 \smallsetminus A| \leq 1$.
\end{lemma}
\begin{proof}
  Corresponding to a partition $U_1 \sqcup U_2$ is the closure of the
  locus in $\M_V$ in which a node separates the sphere into two lobes,
  one containing the points of $U_1$ and the other containing the
  points of $U_2$.  This closed stratum fibres over
  $\M_{V\smallsetminus A}$ if and only if one of the $U_i$ contains at
  most one point of $V\smallsetminus A$ (a priori it is an unordered
  partition, but we can always choose to call the set with this
  property $U_1$).
\end{proof}

Given a partition $V = U_1 \sqcup U_2$ as above, we will write
$\partial_{U_1,U_2} \subset \M_{V}$ for the corresponding fibrewise
boundary component of $\pi^A$.

\subsection{Orientations}\label{orientations}
Fibrewise orientations of the canonical projections play an important
role in the formality proof in this paper.  Here we establish some
general definitions and conventions for working with orientations.

Given a finite dimensional real vector space $W$, let $\det(W)$ denote 
the top exterior power of $W$; a choice of a ray in this line is 
equivalent to a choice of an orientation of $W$.  If $S$ is a finite 
set then we write $\det(S) \coloneqq \det(\R^S)$ (and $\det(S) = \R$ when $S$ 
is empty).  If $S$ has cardinality at least 2 then choosing a ray in
$\det(S)$ is equivalent to choosing an ordering of $S$ up to even 
permutation. 

An orientation form on an $n$-dimensional manifold $M$ is a nowhere
vanishing $n$-form $\Omega_M$.  Clearly an orientation form determines
an orientation of $M$, and two orientation forms determine the same
orientation if and only if they agree up to multiplication by a
positive scalar function.  Given a fibre bundle $\pi: E \to B$, a
fibrewise orientation form is a form on $E$ whose
restriction to each fibre is an orientation form.

We will use the following general orientation conventions.
\begin{itemize}
\item Given $x\in S$, we define an isomorphism $\iota_x\colon \det(S) \cong \det(S
  \smallsetminus \{x\})$ by the rule $x \wedge r \mapsto r$ (inverse
  to the map $(x\wedge -)$).  
\item If $M$ is a manifold with boundary with orientation form
  $\Omega_M$ and $X$ is an outward pointing normal vector field on
  $\partial M$, then $\partial M$ is given the induced orientation
  $\Omega_{\partial M} = \iota_X \Omega_M$.  Thus the Stokes formula
  is $\int_M d\xi = \int_{\partial M} \xi$.
\item If $\pi_1\colon E_1 \to E_2$ and $\pi_2\colon E_2 \to B$ are fibre bundles
  with fibrewise orientation forms $\Omega_1$ and $\Omega_2$
  respectively, then we give the the composite bundle $(\pi_2 \circ
  \pi_1)$ the fibrewise orientation form $\pi_1^*\Omega_2 \wedge \Omega_1$.
\end{itemize}

Given a finite set $V$ and an element $u\in V$, the canonical
projection $\pi^{\{u\}} \co \M_V \to \M_{V\smallsetminus \{u\}}$ has a
standard fibrewise orientation defined as follows.  It factors as
\[
\M_V \stackrel{\pi_{\mathrm{dec}}^{\{u\}}}{\longrightarrow} \M_V/S^1
\stackrel{\pi^{\{u\}}_{\mathrm{pos}}}{\longrightarrow}
\M_{V\smallsetminus \{u\}},
\]
where $\pi^{\{u\}}_{\mathrm{dec}}$ forgets the tangent ray decoration
at the marked point $u$, and $\pi^{\{u\}}_{\mathrm{pos}}$ forgets the
position of $u$, thus forgetting it entirely.  The principal
$S^1$-bundle $\pi^{\{u\}}_{\mathrm{dec}}$ inherits a fibrewise
orientation $d\theta_u$ from the counterclockwise orientation of the
circle.  The map $\pi^{\{u\}}_{\mathrm{pos}}$ is a family of Riemann
surfaces and so the complex structures of the fibres determine a
fibrewise orientation $\Omega_u$. Using the above orientation
conventions, we then give $\pi^{\{u\}}$ the fibrewise orientation
determined by the form $(\pi^{\{u\}}_{\mathrm{dec}})^*\Omega_u \wedge d\theta_u$.

Given $A \subset V$, we define a bijection between rays in
$\det(A)$ and fibrewise orientations of $\pi^A\co
\M_{V} \to \M_{V \smallsetminus A}$ as follows.  The ray spanned by $u_1
\wedge \cdots \wedge u_k$ (where $u_i$ are the elements of $A$)
corresponds to the fibrewise orientation determined by the above
conventions and the factorisation $\pi^{A}=\pi^{\{u_1\}} \circ \cdots
\circ \pi^{\{u_k\}}$.

\section{Cohomology of the unframed and framed little discs}
We will first recall Arnold's well-known presentation of the
cohomology of the ordered configuration spaces of $\R^2$, or
equivalently of the spaces of the unframed little discs operad
$D_2$.  Then we will use this result to give a new
presentation of the cohomology of the framed little discs operad which
will be convenient for working with the cyclic operad structure.

\subsection{Unframed little 2-discs}
Consider the little $2$-discs operad $D_2$. 
For $1\leq i< j \leq k$,
consider the map
\[
\pi_{ij}' \colon D_2(k) \to D_2(2) \simeq S^1
\]
which forgets all but the $i^{th}$ and $j^{th}$ discs.  Define a
collection of 1-forms
\[
\omega_{ij} \coloneqq (\pi_{ij}')^* \mathit d\theta \, ,
\]
where $d\theta$ is the standard volume form of $S^1$.  (We will often
abuse notation by using the same symbol for both a differential form
and the cohomology class it represents.)

\begin{theorem}{\rm (Arnold)}\label{Arnold-pres}
  The real cohomology of $D_2(n)$ is generated by the degree 1 classes
  $\omega_{ij}$ for $1\leq i \neq j \leq n$, subject to only to the symmetry relation,
  \[
  \omega_{ij}=\omega_{ji},
  \]
  and the Arnold relation,
  \[ 
    \omega_{ij} \omega_{jk} + \omega_{jk} \omega_{ki} +
    \omega_{ki} \omega_{ij} = 0
  \]
  for each triple of distinct indices $\{i,j,k\}$.
\end{theorem}
In fact, this gives a presentation of the integral cohomology, but we
shall only be concerned with real cohomology in this paper.

\subsection{Framed little 2-discs}

Since $\M_{n+1} \simeq fD_2(n) = D_2(n)\times (S^1)^{n}$, where
the $i^{th}$ $S^1$ factor measures the angle by which the $i^{th}$
disc is rotated, one has
\begin{equation}\label{framed-noncyclic-pres}
H^*(\M_{n+1}) \cong H^*(D_2(n)) \otimes \Lambda(\eta_1,
\ldots, \eta_n),
\end{equation}
where $\eta_i$ is the class represented by the volume form $d\theta_i$
of the $i$-th circle.  Thus the cohomology is generated by the classes
$\omega_{ij}$ together with the classes $\eta_i$, and subject only to
the symmetry and Arnold relations.  The symmetric group $\Sigma_{n+1}$ acts on
$\M_{n+1}$ by permuting the labels of the marked points.  The subgroup
$\Sigma_n\subset \Sigma_{n+1}$, consisting of permutations which fix
0, permutes the generators and relations among themselves in the
obvious way.  However, the permutations which do not fix zero act in a
more complicated way.

We will now give a different presentation in which the generators and
relations of $H^*(\M_{n+1})$ are both permuted by the full symmetric
group $\Sigma_{n+1}$.  Let $\pi_{ij}\colon \M_{n+1} \to \M_2$ the
projection forgetting all but the $i$-th and the $j$-th points with
their tangent rays (see earlier section \ref{canonical-projections}).
Since $\M_2 \cong S^1$, we define closed 1-forms
\[
\alpha_{ij} \coloneqq \pi_{ij}^* d\theta \in \Omega^*(\M_{n+1}).
\]
which represent generators of the cohomology.  To describe the
relations among these generators we must relate them to the generators
$\omega_{ij}$, $\eta_i$ described above.

There are two ways of identifying the space $\M_{3}$ with $(S^1)^3$.
The first way is to apply the unique conformal automorphism of $\C
P^1$ which puts the marked points $p_0, p_1, p_2$ labelled by $0,1,2$
at the points $\infty, 0, 1 \in \C P^1$ respectively.  One then reads
off the angle $\varphi_i$, in the counter-clockwise sense, from the
equator (going in the direction from $0$ to $1$ to $\infty$) to the
tangent ray at the $i$ marked point.  In terms of these coordinates,
$\alpha_{ij} = d(\varphi_i +\varphi_j)$.

The other way of identifying $\M_3$ with $(S^1)^3$ is to put the
points at $\infty,0,1$ and then rotate the sphere around the axis
through $0$ and $\infty$, so that the tangent ray of the point at
$\infty$ is parallel to the equator (pointing in the direction from
$\infty$ to $0$ to $1$).  We define $\theta_1$ to be the angle between
the ray at $0$ and the equator. Parallel transport along the geodesic
from $p_1$ to $p_2$ of the tangent ray to the equator at $p_1$ gives a
reference ray at $p_2$ and $\theta_2$ is the angle from this ray to
the tangent ray of $p_2$.  The angle $\psi$ is measured at $p_1$ from
the equator to the geodesic from $p_1$ and $p_2$.  In terms of
these coordinates, $\omega_{12} = d\psi$ and $\eta_i = d\theta_i$.

These two different coordinate systems are illustrated below.
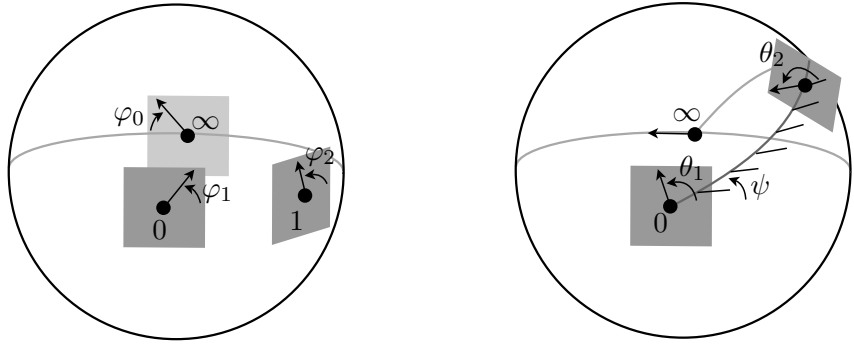
\begin{center}
\input{coord-systems.pst}
\end{center}
One easily sees that these two coordinate systems are related as follows:
\begin{align*}
\theta_1 & = \varphi_1 + \varphi_0, & \varphi_1 = \theta_1 - \psi \\
\theta_2 & = \varphi_2 + \varphi_0, & \varphi_2 = \theta_2 - \psi \\
\psi & = \varphi_0. & 
\end{align*}
Hence, in the cohomology of $\M_n$ the two sets of generators are
related by:
\begin{align*}
\omega_{ij} & = (1/2)(\alpha_{0i} + \alpha_{0j} - \alpha_{ij}), & 
\alpha_{ij} = \eta_i + \eta_j - 2\omega_{ij} \\
\eta_i & = \alpha_{0i}. 
\end{align*}

In terms of the $\alpha$ classes, the Arnold relation becomes:
\begin{align*}
  &\alpha_{ij} \alpha_{jk} + \alpha_{jk} \alpha_{ki} + \alpha_{ki}\alpha_{ij}   \\
 +&\alpha_{i0} \alpha_{0j} + \alpha_{j0} \alpha_{0k} + \alpha_{k0}\alpha_{0i}   \\
 +&\alpha_{kj} \alpha_{j0} + \alpha_{ik} \alpha_{k0} + \alpha_{ji}\alpha_{i0}   \\
 +&\alpha_{0j} \alpha_{ji} + \alpha_{0k} \alpha_{kj} + \alpha_{0i}\alpha_{ik} = 0.
\end{align*}
The action of the symmetric group must send relations to relations,
so one can write the relations in a more invariant form.
Let $\{i_1,i_2,i_3,i_4\}$ be four distinct elements of $\{0,\ldots n\}$.
The \emph{cyclic Arnold relations} are
\[
\sum_{\sigma \in A_4} \alpha_{i_{\sigma(1)} i_{\sigma(2)}}
\alpha_{i_{\sigma(2)} i_{\sigma(3)}} = 0,
\]
where $A_4$ is the alternating group on 4 letters.  Clearly the
symmetric group $\Sigma_{n+1}$ sends each relation of this type to
another relation of this type. 

\begin{theorem}\label{cyclic-pres}
  The real cohomology of $\M_{n+1}$ is generated by the 1-dimensional
  classes $\alpha_{ij}$ ($0 \leq i \neq j \leq n$) subject only to the
  cyclic Arnold relations, one for each 4-tuple of distinct indices
  $\{i_1,i_2,i_3,i_4\}$ between $0$ and $n$, and to the symmetry
  relation $\alpha_{ij}=\alpha_{ji}$.
\end{theorem}

\section{The affine and projective graph complexes}

The main ingredient in Kontsevich's proof of formality of the unframed
little discs operad is the construction of a certain differential
graded algebra of graphs, the \emph{Kontsevich graph complex} $\kont_n$,
which resolves the usual Arnold presentation (Theorem
\ref{Arnold-pres}) of the cohomology of the little discs $D_2(n)$.  As
$n$ varies these DGAs of graphs fit together to form a cooperad, and
the dual operad is quasi-isomorphic to the operad of chains on the little
discs.

Tensoring Kontsevich's resolution $\kont_n$ of $H^*(D_2(n))$ with
$\Lambda(\eta_1, \ldots, \eta_n)$ gives a resolution of the
presentation \eqref{framed-noncyclic-pres} of the cohomology ring of
$fD_2(n) \simeq \M_{n+1}$ --- we shall refer to this DGA as the
\emph{affine graph complex}, denoted $\aff_n$.  The collection of DGAs
$\{\aff_n\}$ forms a cooperad, but it is \emph{not} compatible with
the \emph{cyclic} operad structure of $\M$.  To deal with this we will
pass to the \emph{projective graph complex} $\proj_{n+1}$.  The
projective graph complexes do form a cyclic cooperad and we prove that
$\proj_{n+1}$ is indeed a resolution of $H^*(\M_{n+1})$ by a comparison with
$\aff_n$.  Summarising this:
\begin{itemize}
\item $\kont_n$ is Kontsevich's graph complex; it resolves $H^*(D_2(n))$ 
over $\R[\Sigma_n]$.
\item $\aff_n = \kont_n \otimes \Lambda (\eta_1, \ldots, \eta_n)$ is
  the affine graph complex; it resolves $H^*(\M_{n+1})$ over $\R[\Sigma_n]$.
\item $\proj_{n+1}$ is the projective graph complex, which 
  resolves $H^*(\M_{n+1})$ over $\R[\Sigma_{n+1}]$.
\end{itemize}
Each of these complexes will be defined in detail below.

\subsection{Some preliminaries on graphs and orientations of graphs}

By a graph $\gamma=(V,V_{int}, E)$ we shall mean a finite set $V\sqcup
V_{int}$ of vertices and a set $E$ of unordered pairs of vertices
representing the edges.  Note that our graphs cannot have loops or
double edges.  Those vertices in $V_{int}$ are called \emph{internal
  vertices}, and those in $V$ are called \emph{external vertices};
there is an induced partition of the set of edges
\[
E=E_{ext} \sqcup E_{\partial} \sqcup E_{int},
\]
where the set of \emph{internal edges}, $E_{int}$, consists of those
edges for which both endpoints are internal, the set of \emph{external
  edges}, $E_{ext}$, consists of those for which both endpoints are
external vertices, and the set of boundary edges, $E_\partial$,
consists of those which have one internal endpoint and one external
endpoint.

\begin{definition}
  An \emph{affine orientation} of a graph $\gamma=(V, V_{int},E)$ is a choice
  of a ray in the real line $\det(E)$.  A \emph{projective
    orientation} of $\gamma$ is a choice of a ray in
  $\det(V_{int} \sqcup E)$.
\end{definition}
If $\gamma$ is equipped with an affine or projective orientation then
we shall write $\overline{\gamma}$ for the same underlying graph
equipped with the opposite orientation.

\subsubsection{Edge deletion}
Given a graph $\gamma=(V,V_{int},E)$ and an edge $e$, we define a new
graph $\gamma\smallsetminus e = (V,V_{int}, E\smallsetminus \{e\})$ by
deleting the edge $e$.  If $\gamma$ has a projective or affine
orientation, then $\gamma \smallsetminus e$ inherits an orientation of
the same type by the isomorphism $\iota_e$ from section \ref{orientations}. 

\subsubsection{Edge contraction}
Given a graph $\gamma=(V,V_{int},E)$ and an edge $e$ that is not part
of a triangle of edges and has at least one internal endpoint, let
$\gamma/e$ denote the graph constructed from $\gamma$ by contracting
the edge $e$, i.e. deleting $e$ from the set of edges and identifying
its two endpoints together.  If $e$ is a boundary edge then the
resulting vertex is external; if $e\in E_{int}$ then the resulting
vertex is internal.  Thus the set of external vertices of $\gamma/e$
is identified with $V$.

If $\gamma$ has an affine orientation then $\gamma/e$ inherits an
affine orientation by the isomorphism $\iota_e$.  Inducing projective
orientations is slightly more complicated and depends on whether $e$
is boundary, or internal.
\begin{itemize}
\item If $e \in E_{\partial}$ then one endpoint, $u$, is external and one
  endpoint, $v$, is internal.  In this case there is a canonical
  identification of the edges of $\gamma/e$ with $E\smallsetminus
  \{e\}$; we identify the external vertices of $\gamma/e$ with $V$,
  and the internal vertices with $V_{int} \smallsetminus \{v\}$, by
  identifying the vertex in $\gamma/e$ at which $e$ was contracted
  with $u$.  Using these identifications, the induced projective
  orientation of $\gamma/e$ is the image of the projective orientation
  of $\gamma$ under the isomorphism $\iota_v \circ \iota_e$.
\item If $e\in E_{int}$ then inducing a projective
  orientation on $\gamma/e$ is equivalent to choosing an orientation
  of $e$, meaning a choice of one end, $v$, as \emph{head} and the
  other, $u$ as \emph{tail}.  As in the boundary edge case, we
  contract $e$ onto its tail and identify the internal vertices of
  $\gamma/e$ with $V \smallsetminus \{v\}$, and then we induce a
  projective orientation via $\iota_v \circ \iota_e$.
\end{itemize}

\subsubsection{Combined deletion and contraction}
Suppose that $(e,f)$ is a pair of edges in a graph $\gamma$ meeting
at an internal vertex $u$ and such that $e$ is not a side of a
triangle in $\gamma\smallsetminus f$. Then we shall write
\[
\gamma \oslash (e,f) \coloneqq (\gamma \smallsetminus f)/e.
\]
If $\gamma$ has a projective orientation and $e$ is either a boundary
edge or an oriented internal edge, then $\gamma\oslash (e,f)$ inherits a projective
orientation by the above discussion.

\subsection{The affine and projective graph complexes}

First recall Kontsevich's complex of ``admissible graphs''
\cite[Definition 13]{Ko}, denoted $\kont$.
\begin{definition}
  The vector space $\kont_V$ is spanned by isomorphism classes of
  affine oriented graphs with external vertex set $V$ modulo the relations
  \begin{enumerate}
  \item $\gamma \sim 0$ if there is a connected component of $\gamma$
    containing no external vertices,
  \item $\gamma \sim 0$ if there is an internal vertex of valence $<
    3$,
  \item $\overline{\gamma} \sim -\gamma$.
  \end{enumerate}
  There is a grading defined by $|E| - 2|V_{int}|$.  The differential
  of degree 1 is given by the formula
  \[
  d\gamma = \sum_{e } \gamma / e, 
  \]
  where $e$ runs among those boundary and internal edges such that
  $\gamma/e$ is defined.  There is an associative and graded
  commutative product defined by gluing the external vertices together
  according to the labelling, and by tensoring orientation rays.
\end{definition}

Note that the differential is a derivation of the product, and so
$\kont_V$ forms a commutative DGA.  We will write $\kont_V = \kont_n$ when
$V=\{1,\ldots, n\}$.

\begin{definition} \label{agc} The affine graph complex is the CDGA
  defined as \[\aff_n \coloneqq \kont_n \otimes \Lambda(\eta_1,\ldots,
  \eta_n).\]
\end{definition}

We now come to principal object of this paper: the CDGA of
\emph{projective graphs}.
\begin{definition} \label{defproj}
  The vector space of projective graphs $\proj_V$, for a finite set
  $V$, is spanned by isomorphism classes of projectively oriented
  graphs $\gamma$ with $V$ as set of external vertices, 
  modulo the relations
\begin{itemize}
\item[$(R1)$] if $\gamma$ has an internal vertex of valence $\leq 3$ then $\gamma \sim 0$,
\item[$(R2)$] if $\gamma$ has a component with at most one external vertex and at least one internal vertex then $\gamma \sim 0$,
\item[$(R3)$] $\overline{\gamma} \sim -\gamma$,
\item[$(R4)$] the ``pinwheel relation'': for $v$ an internal vertex of
  $\gamma$ and $E(v)$ the set of edges incident at $v$,
\[
\sum_{e\in E(v)} (\gamma \smallsetminus e) \sim 0.
\]
\end{itemize}
There is a grading defined by $\mathrm{deg} (\gamma) \coloneqq |E|-3|V_{int}|$.
There is an associative and graded commutative algebra
structure on $\proj_V$ given by gluing at the external
vertices, and tensoring orientation rays.
\end{definition}
In the case $V=\{0,\dots,n-1\}$, we write $\proj_V = \proj_n$.

We next define a differential $d$ on $\proj_V$.  The formula will
involve choices, but we will show that because of the pinwheel
relation it is in fact independent of these choices.  To define $d
\gamma$, with $\gamma \in \proj_V$, we first arbitrarily choose an
orientation of each internal edge $e$ of $\gamma$.  We orient the
boundary edges in each case so that their heads are at their internal
endpoints. The differential is now given by the
formula
\[
d \gamma = \sum_{(e, f)} \gamma \oslash (e,f)
\]
where the sum runs over all ordered pairs of edges $(e,f)$ for which
the head of $e$ is internal and incident to $f$ (so in particular,
they must lie in $E_\partial \sqcup E_{int}$) and such that $\gamma
\oslash (e,f)$ is defined (i.e. deleting $f$ and contracting $e$ does
not result in a loop or double-edge).

\begin{proposition}
$d\gamma$ is well-defined, independent of the choices of orientations
for the edges between internal vertices. 
\end{proposition}
\begin{proof}
Suppose $e$ is an edge between internal vertices $v$, $w$.  
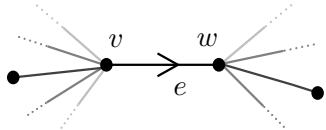
\begin{center}
\input{diff-well-defined.pst}
\end{center}
The difference between $d\gamma$ with $e$ oriented one way or the
other is precisely the pinwheel relation in $\gamma/e$ at the vertex
where $e$ is contracted.  Thus $d\gamma$ is independent of the choices
of orienting the internal edges.  

It remains to check that $d$ descends to the quotient by the pinwheel
relation.  Given a graph $\gamma$ with an internal vertex $v$, orient
all of the (internal) edges of $E(v)$ with their heads at $v$. 
Then,
\begin{align*}
d\Bigg( \sum_{g \in E(v)} & \gamma \smallsetminus g \Bigg)  = 
\sum_{g \in E(v)} \:\sum_{(e,f) \mathrm{\:in\:} \gamma \smallsetminus g}
 (\gamma \smallsetminus g)\oslash (e,f) \\
& = 
\sum_{g \in E(v)} \: \sum_{\substack{(e,f) \mathrm{\:in\:} \gamma
    \smallsetminus g \\ e \notin E(v)}}
 \hspace{-4mm}  -(\gamma \oslash (e,f) \smallsetminus g) \quad + \quad  
\sum_{g \in E(v)} \: \sum_{\substack{(e,f) \mathrm{\:in\:} \gamma
    \smallsetminus g \\ e \in E(v)}} \hspace{-2mm} - (\gamma \oslash (e,f) \smallsetminus g).
\end{align*}
The first of the two above summations can be grouped as a sum of terms, each of which is a
pinwheel sum over edges $g$ incident at $v$ in $\gamma\oslash (e,f)$.
The terms of the second summation cancel in pairs since $e\in E(v)$
oriented with head at $v$ means that $f\in E(v)$, and $ \gamma \oslash
(e,f) \smallsetminus g = -\gamma \oslash (e,g) \smallsetminus f$.
\end{proof}

\begin{proposition}
$d^2 = 0$.
\end{proposition}
\begin{proof}
  To compute $d^2 \gamma$ we first choose an orientation of each internal edge of
  $\gamma$; $d\gamma$ is a sum of terms of the form
  $\gamma\oslash(e,f)$ and the internal edges of each of these terms
  inherit orientations.  Now, $d^2 \gamma$ is a sum of graphs of the form $\pm
  \gamma\oslash(e,f)\oslash(e',f')$ (the sign depends only on
  $\gamma$), where $e,f$ are internal or boundary edges in $\gamma$
  with $f$ incident at the head of $e$, and $e',f'$ are internal or
  boundary edges in $\gamma\oslash(e,f)$ with $f'$ incident at the
  head of $e'$.  This means that, as edges in $\gamma$, the
  possible configurations are enumerated as follows:

  Type (0): the head of $e'$ is incident at $f'$ in $\gamma$ and
  the heads of $e$ and $e'$ are disjoint.

  If the head of $e'$ is incident at $f'$ in $\gamma$ but the
  configuration is not of type (0) then it must be

  Type (1): $f$, $f'$,  $e$ and $e'$ all meet at a
  vertex.
  
  If the head of $e'$ is not incident at $f'$ in $\gamma$ then it must
  be incident at $f'$ in $\gamma\oslash (e,f)$.  This can happen only
  if, in $\gamma$, either $e'$ and $f'$ are incident at opposite ends
  of $e$ (types (2), (4), and (5) below), or if $f'$ is incident at
  the tail of $e'$ and contracting $e$ causes $e'$ to reverse
  orientation (type (3) below).  So the configuration must be one of
  the following types.
  \begin{center}
    \input{allcases.pst}
  \end{center}

  We will show case by case that all terms can be grouped to exactly
  cancel.  For type (0), the way that orientations are induced yields 
 \[
  \gamma\oslash (e,f) \oslash (e',f') = \overline{\gamma\oslash
    (e',f') \oslash (e,f)},
  \]
  and so these two terms in $d^2\gamma$ cancel.  For types (1) and (2),
  $\gamma\oslash (e,f) \oslash (e',f')$ cancels with $\gamma\oslash
  (e,f') \oslash (e',f)$.  In the configuration of type (3), $\gamma \oslash
  (e,f) \oslash (e',f')$ cancels with $\gamma\oslash (e',f) \oslash
  (e,f')$ which is a term of type (4), and hence the terms of type (3)
  cancel in pairs with the terms of type (4).

  Type (5) is slightly more involved.  By the pinwheel relation,
  \begin{align}\label{pinwheel-cancelation-1}
  \sum_{f' \in E(u) \smallsetminus e} \gamma\oslash(e,f) \oslash
  (e'f') \quad = &\quad
  - \sum_{g\in E(v) \smallsetminus \{e,f,e'\}} 
         \gamma \oslash (e,f) \oslash (e',g) \\
  & \quad + \sum_{h\in E(w) \smallsetminus e'} 
         \gamma \oslash (e,f) \oslash (\overline{e'},h). \nonumber
  \end{align}
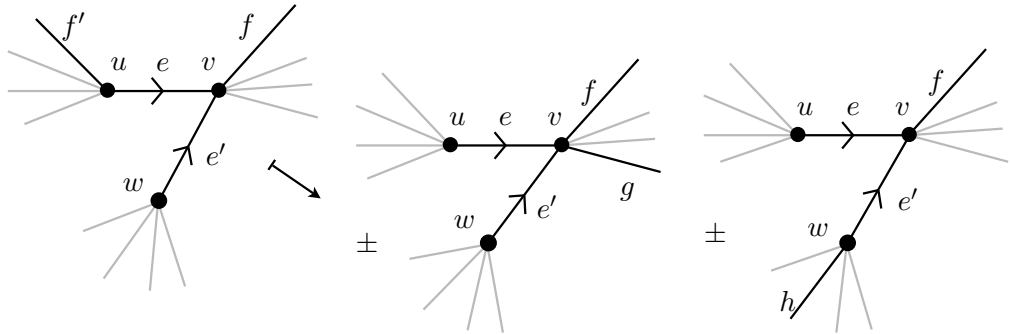
  \begin{center}
  \input{case-head-to-head.pst}
  \end{center}
  Consider the sum of the expression \eqref{pinwheel-cancelation-1}
  over all $f\in E(v) \smallsetminus \{e,e'\}$; the sum over $f,g$ is
  zero because of the symmetry of swapping $f$ and $g$ (just as for
  type (0)), and the sum over $f,h$ exactly cancels with the sum over
  $f,h$ of the type (5) terms $\gamma\oslash (e',f) \oslash (e,h)$
  occurring in $d^2\gamma$.
\end{proof}

\subsection{The combinatorial `integration' maps}
Here we introduce maps of graph complexes that are combinatorial
analogues of certain configuration space integrals.

Fix finite sets $V,V_{int}$.  Let $\proj^0_{V\sqcup V_{int}} \subset
\proj_{V\sqcup V_{int}}$ denote the subalgebra of graphs having no
internal vertices.  Note that for such a graph $\gamma=(V\sqcup
V_{int},\emptyset,E)$, a projective orientation is simply a ray in
$\det(E)$.  Given a ray $r \subset \det(V_{int})$, there is a linear
map
\[
\pi_r^{V_{int}} \colon \proj^0_{V \sqcup V_{int}} \to \proj_V.
\]
It is defined by turning the $V_{int}$ vertices into internal vertices
and mapping the orientations by
\[
\det(E) \stackrel{-\wedge r}{\longrightarrow} \det(V_{int}\sqcup E).
\]
Note that $\pi^{V_{int}}_{-r} = - \pi^{V_{int}}_r.$  Geometrically, this map
corresponds to integrating out the $V_{int}$ vertices as we shall see
in section \ref{int-section}.

A general projectively oriented graph $\gamma \in \proj_V$ with set of
internal vertices $V_{int}$ can be written as $\gamma
=\pi_r^{V_{int}}(\widetilde{\gamma})$, where $\widetilde{\gamma} \in
\proj^0_{V\sqcup V_{int}}$ and $r$ is a ray in $\det(V_{int})$. 

\subsection{The cyclic cooperad structure on the projective graph complexes}\label{cooperad-structure}

We now define a cyclic cooperad structure on the collection of
projective graph complexes.  For any partition $V=I\sqcup J$ we need
to define a co-composition map
\[
_I\bullet_J \colon \proj_V \to \proj_{I\sqcup\{x_1\}} \otimes
\proj_{J\sqcup\{x_2\}}.
\]
Let $e_{a,b}$ be the graph consisting of a single
edge between external vertices $a$ and $b$, oriented from $a$ to $b$.  Then,
\[
(_I\bullet_J)(e_{a,b}) = 
\begin{cases}
e_{a,b} \otimes 1  & \mbox{if $a,b \in I$,} \\
1\otimes e_{a,b}   & \mbox{if $a,b \in J$,} \\
e_{a,x_1}\otimes 1 + 1 \otimes e_{x_2,b} & \mbox{if $a \in I$ and $b\in J$.}
\end{cases}
\]

A graph $\gamma \in \proj^0_{V}$ can be written as a product of graphs
each having only a single edge.  The order of the factors corresponds
to the chosen orientation.  The multiplication of the formulae above
for the edges of $\gamma$ defines $(_I \bullet_J)(\gamma)$.

Given a partition $V_{int} = A \sqcup B$, choose rays $s \in \det(A)$
and $t \in \det(B)$ such that $s \wedge t = r$.  Now consider the
composition
\[
\phi^{I,J}_{A,B} \colon \proj^0_{V \sqcup V_{int}} 
\xrightarrow{_{I\sqcup A}\bullet_{J\sqcup B}} \proj^0_{I
  \sqcup A \sqcup \{x_1\}} \otimes \proj^0_{J \sqcup B \sqcup \{x_2\}}
\xrightarrow{\pi_s^A \otimes \pi_{t}^B} \proj_{I\sqcup
  \{x_1\}} \otimes \proj_{J\sqcup \{x_2\}}.
\]
We can now define the cooperad co-composition map on an arbitrary
graph $\gamma = \pi_r^{V_{int}}(\widetilde{\gamma})$ by the formula
\begin{equation}\label{co-comp1}
(_I \bullet_J)(\gamma) \coloneqq \sum_{V_{int}=A \sqcup B}
\phi^{I,J}_{A,B}(\widetilde{\gamma}).
\end{equation}
There is a useful alternative description of this co-composition that
will make the verification of certain properties easier.  An
$(I,J)$-\emph{splitting} of a graph $\gamma$ is a partition of the
vertices into two sets $V_L \sqcup V_R = V \sqcup V_{int}$ with $I
\subset V_L$ and $J \subset V_R$ and a partition of the edges into two
sets $E_L \sqcup E_R$ such that if both of the endpoints of an edge
$e$ are in $V_i$ ($i\in \{L,R\}$) then $e\in E_i$.  Given a splitting
$\tau$, we can construct two new graphs $\gamma^\tau_i$ with vertices
$V_i \sqcup \{x_i\}$ and edges $E_i$; if $e\in E_i$, as an edge in
$\gamma$, has only one endpoint in $V_i$ then as an edge in
$\gamma^{\tau}_i$ it's endpoint that is not in $V_i$ is replaced by
the auxiliary vertex $x_i$.  If $\gamma$ has a projective orientation
$\Omega$ then we give $\gamma^\tau_L$ and $\gamma_R$ projective
orientations $\Omega_L$ and $\Omega_R$ such that $\Omega_L \wedge
\Omega_R = \Omega$.  By inspection one finds the following.
\begin{proposition}
  The co-composition map defined above in \eqref{co-comp1} is given by
  the formula
  \[
  _I \bullet_J (\gamma) = \sum_{\tau \in S(\gamma)} \gamma^\tau_L \otimes \gamma^\tau_R,
  \]
  where the sum runs over the set $S(\gamma)$ of all $(I,J)$-splittings of $\gamma$.
\end{proposition}

\begin{proposition}
  The co-composition maps given above are well-defined and they make
  $\proj_V$ into a cyclic cooperad in the category of commutative
  differential graded algebras.
\end{proposition}
\begin{proof}
  There are three facts to verify.
  \begin{itemize}
  \item[(i)] $_I \bullet_J$ is compatible with the relations $(R1-4)$ of Definition \ref{defproj}. 
    \item[(ii)] $_I \bullet_J$ commutes with the differential.
  \item[(iii)] These maps satisfy the cyclic cooperad
    axioms given in section \ref{cyclic-operad-definitions}.
  \end{itemize}
  For (i), it is easy to verify compatibility with $(R1-3)$. For the
  pinwheel relation $(R4)$, suppose $\gamma$ is a projective graph
  with an internal vertex $v$ and external vertices $I\sqcup J$.  Then
\begin{align}\label{co-comp-pinwheel1}
{_I \bullet_J} \left(\sum_{f\in E(v)} \gamma \smallsetminus f \right) & = \sum_{f \in
  E(v)} {_I \bullet_J} (\gamma \smallsetminus f) \nonumber \\
 & =  \sum_{f \in E(v)} \sum_{\tau} \gamma^\tau_L \otimes \gamma^\tau_R, 
\end{align}
where $\tau$ runs over the set $S(\gamma \smallsetminus f)$ of
$(I,J)$-splittings of $\gamma \smallsetminus f$.  
Then \eqref{co-comp-pinwheel1} can be written as
\begin{align*}
  & \sum_{f \in E(v)} \left( \sum_{\substack{\tau \in S(\gamma
        \smallsetminus f) \\ \mathrm{s.t.\:} v\in V_L}} \gamma^\tau_L
    \otimes \gamma^\tau_R \quad +\quad \sum_{\substack{\tau \in
        S(\gamma \smallsetminus f) \\ \mathrm{s.t.\:} v\in V_R}}
    \gamma^\tau_L \otimes
    \gamma^\tau_R \right)\\
  = & \sum_{f \in E(v)} \left( \sum_{\substack{\tau' \in S(\gamma) \\
        \mathrm{s.t.\:} v\in V_L\\ \mathrm{and\:} f\in E_L}} (\gamma_L^{\tau'}
    \smallsetminus f) \otimes \gamma^{\tau'}_R \quad + \quad
    (-1)^{\mathrm{deg}(\gamma^{\tau'}_L)}\hspace{-4mm} \sum_{\substack{\tau' \in
        S(\gamma) \\ \mathrm{s.t.\:} v\in V_R\\ \mathrm{and\:} f\in E_R}}
    \gamma_L^{\tau'} \otimes (\gamma^{\tau'}_R \smallsetminus f)
  \right) \\
  = & \sum_{\substack{\tau' \in S(\gamma) \\ \mathrm{s.t.\:} v \in
     V_L}} \sum_{f \in E_L\cap E(v)}  \hspace{-4mm} (\gamma^{\tau'}_L \smallsetminus f)
  \otimes \gamma^{\tau'}_R \quad + \quad
  (-1)^{\mathrm{deg}(\gamma^{\tau'}_L)}  \hspace{-3mm} \sum_{\substack{\tau' \in
     S(\gamma) \\ \mathrm{s.t.\:} v \in V_R}} \sum_{f \in E_R \cap E(v)} \hspace{-4mm}
 \gamma^{\tau'}_L \otimes (\gamma^{\tau'}_R \smallsetminus f).
\end{align*}
This verifies statement (i).

We turn to (ii) now.  Given an oriented edge $e$, let $h(e)$
denote the head and $t(e)$ the tail.  Observe that choosing
a splitting of $\gamma\oslash(e,f)$ is the same as choosing a
splitting of $\gamma$ such that $f,h(e),t(e)$ are either all on the
left or all on the right.  Hence,
\begin{align*}
{_I \bullet_J}(d\gamma) = & \sum_{(e,f) \mathrm{\: in \:} \gamma}
\left( \sum_{\tau \in S(\gamma \oslash (e,f))} \gamma^\tau_L \otimes
\gamma^{\tau}_R \right)\\
=  \sum_{(e,f) \mathrm{\: in \:} \gamma} & \left( \sum_{\substack{\tau' \in
    S(\gamma) \\ \mathrm{s.t.\:} e,f \in E_L\\ h(e),t(e) \in V_L}}
 \hspace{-4mm}\left( \gamma^{\tau'}_L \oslash (e,f) \right) \otimes \gamma^{\tau'}_R 
\quad  + \quad   (-1)^{\mathrm{deg}(\gamma_L^{\tau'})} \hspace{-6mm}\sum_{\substack{\tau' \in
    S(\gamma) \\ \mathrm{s.t.\:} e,f \in E_R\\ h(e),t(e) \in V_R}} \hspace{-4mm}
\gamma^{\tau'}_L \otimes \left( \gamma^{\tau'}_R \oslash (e,f) \right)
\right) .
\end{align*}
Applying the differential to a co-composition and expanding similarly,
one sees that $d( {_I \bullet_J}(\gamma))$ is given by the above sum
but without the requirement $t(e)$ be on the same side as $h(e)$.  We
will show that these additional terms cancel out.  Suppose we are
given a pair $(e,f)$ in $\gamma$ meeting at an internal vertex
$v=h(e)$ and a splitting $\tau\in S(\gamma)$ such that $t(e) \in V_L$,
$h(e) \in V_R$, and $e,f\in E_R$.  Define a new splitting $\tau' \in
S(\gamma)$ by switching $e,f$ and $h(e)$ from the right to the left.
Then
\[
\left(\gamma_L^{\tau'} \oslash (f,e) \right) \otimes \gamma^{\tau'}_R
= (-1)^{\mathrm{deg}(\gamma^\tau_L) + 1} \left(\gamma_L^\tau \otimes \left(\gamma_R^\tau\oslash(e,f) \right)\right),
\]
and so these two terms exactly cancel.  All of the additional terms in
$d({_I \bullet_J}(\gamma))$ can thus be paired in this way to cancel.

Finally, the verification of (iii) is entirely straightforward.
\end{proof}

\section{$\proj_n$ is a resolution of $H^*(\M_n)$}

Kontsevich showed that the map $\kont_{n-1} \to H^*(D_2(n-1))$,
defined by sending a graph with no internal vertices to the
corresponding product of Arnold classes $\omega_{ij}$, one for each
edge $e_{i,j}$, and sending all graphs with internal vertices to zero,
is a quasi-isomorphism of cooperads; see \cite[Section 3.3.4]{Ko} or
\cite[Theorem 9.1]{LambVol}.  It follows immediately that the
analogous map $\aff_{n-1} \to H^*(\M_n)$ is also a quasi-isomorphism
of cooperads.  The cohomology $H^*(\M)$ is a cyclic cooperad, but $\aff$
is \emph{not} a cyclic cooperad.  Hence we must pass from affine to
projective graphs.

\begin{theorem}\label{projective-graphs-resolution}
  The map $\proj_V \to H^*(\M_V)$ defined by sending $\gamma$ to zero
  if it has any internal vertices, and otherwise sending it to the
  corresponding product of $\alpha_{ij}$ classes, one for each edge
  $e_{ij}$, is a quasi-isomorphism of cyclic cooperads.
\end{theorem}

We will derive this theorem from the corresponding statement for
$\aff_{n-1}$ by a spectral sequence comparison of $\aff_{n-1}$ and
$\proj_n$.  The first step is to construct a map between $\aff_{n-1}$
and $\proj_n$, and for this we must introduce a couple of auxiliary
graded vector spaces.  Let $\xaff_{n-1}$ denote the graded vector
space covering $\aff_{n-1}$ in which each graph is equipped with an
injective map from the set of its internal vertices to $\N$.  Let $\xproj_n$
denote the graded vector space covering $\proj_n$ in which the
pinwheel relation is \emph{not} imposed and each graph is equipped
with an injective map from the set of its internal vertices to $\N$.

There is an injective linear map 
\[
\widetilde{\psi}\colon \xaff_{n-1} \hookrightarrow \xproj_n
\]
defined by sending $\gamma \otimes (\eta_{i_1} \wedge \cdots \wedge
\eta_{i_k})$ to the graph constructed by adding a disjoint external
vertex $0$, denoted $u_0$, and then successively adding an external
edge $h_j$ between $u_0$ and the external vertex labelled by $i_j$ for
$j=1,\ldots, k$, and then finally by adding a boundary edge $h_v$ from
$u_0$ to each internal vertex $v$.  The affine orientation $x$ of
$\gamma$ determines a projective orientation of
$\widetilde{\psi}(\gamma\otimes (\eta_{i_1} \wedge \cdots \wedge
\eta_{i_k}))$ by the expression
\[
x \wedge h_1 \wedge \cdots \wedge h_k \wedge (h_v \wedge v)_{v \in V_{int}},
\]
where $v$ runs over the internal vertices in some arbitrary order.

The image of $\widetilde{\psi}$ is precisely the vector space spanned
by graphs for which each internal vertex is connected to $u_0$ by an
edge.  One checks that $\widetilde{\psi}$ descends to a linear map
\[
\psi\colon \aff_{n-1} \to \proj_n.
\]
This is a morphism of graded algebras, but it does \emph{not} commute
with the differentials.  However, we have the following.

\begin{lemma}\label{linear-iso}
The map $\psi\colon \aff_{n-1} \to \proj_n$ is a linear isomorphism.
\end{lemma}

The idea of the proof of the above lemma is to explicitly construct an
inverse to $\psi$ by using the pinwheel relation to rewrite any
projective graph as a sum of projective graphs in which each internal
vertex is connected by an edge to the external vertex $u_0$.  We now formalise
this procedure.  For any $k\in \N$, we define two linear operators
\[
P_k, Q_k \colon \xproj_n \to \xproj_n
\]
by specifying what they do on generators.  Given a projective graph
$\gamma \in \xproj_n$, if there is no internal vertex labelled by $k$
then $Q_k(\gamma) = \gamma$; if there is such an internal vertex then
$Q_k$ simply adds an edge $e$ from that vertex to $u_0$ or outputs
zero if there is already such an edge; the new projective orientation
is given by $x \mapsto e\wedge x$.  Observe that the operators
$Q_k$ anticommute.

Let $E(k)$ denote the set of edges in $\gamma$ which are incident at
the internal vertex with label $k$, where $E(k) = \emptyset$ if there
is no such vertex.

If $E(k)$ contains an edge to the external vertex $u_0$ then
$P_k(\gamma) = \gamma$; otherwise
\[
P_k(\gamma) = - \hspace{-2mm}\sum_{e\in E(k)} Q_k(\gamma) \smallsetminus e.
\]
One can think of $P_k$ as using the pinwheel relation to write
$\gamma$ as a sum of graphs that are in the image of $\psi$, although
we have not yet passed from $\xproj$ to $\proj$.

We can now proceed to define an inverse to $\psi$.  Consider the
infinite composition
\[
\widetilde{P}_\infty \coloneqq
( \cdots \circ P_2 \circ P_1) \colon
 \xproj_n \to \xproj_n;
\]
since any graph $\gamma$ has only finitely many internal vertices, all
but finitely many of the $P_k$ act by identity on $\gamma$, and so
this is well defined.    

\begin{lemma}\label{P-ops}
  \begin{itemize}
  \item[(i)] $\widetilde{P}_\infty$ is equivariant with respect to the
    action of the group $\Sigma_\N$ of permutations of $\N$ by
    relabelling the internal vertices.
  \item[(ii)]  $\gamma$ and $P_k(\gamma)$ project to the same element of
  $\proj_n$.
  \item[(iii)] Any two operators $P_h, P_k$ commute.
  \item[(iv)] The pinwheel ideal is contained in the kernel of
    $\widetilde{P}_\infty$.
  \end{itemize}
\end{lemma}
\begin{proof}
  Part (i) is trivial.  Part (ii) is immediate from the pinwheel
  relation.  For part (iii) we need only consider the case where
  $\gamma$ is a graph with internal vertices $j$ and $k$, neither of
  which has an edge to $u_0$ since in all other cases at least one of
  $P_j$ or $P_k$ acts by identity on $\gamma$.  In this situation,
  \[
  P_j \circ P_k (\gamma) = \sum_{e \in E(k)} \sum_{f \in E(j), f\neq e}
  Q_j ( Q_k(\gamma) \smallsetminus e) \smallsetminus f.
  \]
  The operations of deleting edges $e$ and $f$ anticommute, deleting
  either $e$ or $f$ anticommutes with $Q_j$ and $Q_k$, and $Q_j$ and
  $Q_k$ anticommute with one another.  Hence
  \[
  Q_j (Q_k (\gamma) \smallsetminus e) \smallsetminus f = Q_j (Q_k
  (\gamma) \smallsetminus f) \smallsetminus e.
  \]

  We turn to  part (iv).  It suffices to check
  that for any graph $\gamma$ with internal vertex $k$,
  \[
  \widetilde{P}_\infty \left( \sum_{e \in E(k)} \gamma \smallsetminus
    e \right) = 0.
  \]
  Since the operators $P_k$ all commute, we need only show that $P_k$
  sends $\sum_{e \in E(k)} (\gamma \smallsetminus e)$ to zero.  First
  suppose that there is no edge in $\gamma$ between the vertices $k$
  and $u_0$.  Then
  \[
  \sum_{e \in E(k)} P_k(\gamma \smallsetminus e) = - \sum_{e,f \in
    E(k), f\neq e} Q_k(\gamma \smallsetminus e) \smallsetminus f
  \]
  and this is zero since the $(e,f)$ term cancels exactly with the
  $(f,e)$ term.  Now suppose that there is an edge $h$ in $\gamma$
  between the internal vertex labelled $k$ and the external vertex
  $u_0$.  Then
  \begin{align*}
    P_k\left(\sum_{e \in E(k)} \gamma \smallsetminus e\right) & =
    \sum_{e \in E(k), e\neq h} \hspace{-2mm} (\gamma \smallsetminus e)
    \quad + \quad  P_k(\gamma \smallsetminus h)  \\
    & = \sum_{e \in E(k), e\neq h} \hspace{-2mm} (\gamma \smallsetminus e) \quad -
    \sum_{f\in E(k), f\neq h} \hspace{-2mm} Q_k(\gamma \smallsetminus h)
    \smallsetminus f = 0
  \end{align*}
  since $Q_k(\gamma \smallsetminus h) = \gamma$ so all terms cancel.
\end{proof}

\begin{proof}[Proof of Lemma \ref{linear-iso}]
  Observe that $\widetilde{P}_\infty(\gamma)$ lies in the image of
  $\widetilde{\psi}$.  Since $\widetilde{\psi}$ is injective we
  can thus regard $\widetilde{P}_\infty$ as a linear map
  \[
  \widetilde{P}_\infty\colon \xproj_n \to \xaff_{n-1}
  \]
  which clearly satisfies $\widetilde{P}_\infty \circ
  \widetilde{\psi} = \mathrm{id}_{\xaff_{n-1}}$ and, by Lemma \ref{P-ops},
  descends to a map
  \[
  P_\infty\colon \proj_n \to \aff_{n-1}
  \]
  such that $P_\infty \circ \psi = \mathrm{id}_{\aff_{n-1}}$.  On the
  other hand, Lemma \ref{P-ops} implies that $\widetilde{P}_\infty$
  descends to the identity on $\proj_n$.  But the induced map on $\proj_n$
  is equal to $\psi \circ P_\infty$.  Therefore $\psi$ is
  bijective.
\end{proof}

The complex $\proj_n$ has a descending filtration
\[
\proj_n=F_0(\proj_n) \supset F_1(\proj_n) \dots \supset F_{n-1}(\proj_n) \supset 
F_n(\proj_n)=0
\]
in which the $k$-th filtration level, $F_k(\proj_n) \subset \proj_n$, is the
subcomplex spanned by projective graphs having at least $k$ edges
connecting $u_0$ to other external vertices.  The filtration is well
defined because the pinwheel relation does not affect edges between
external vertices, and similarly, the differential is compatible with
the filtration.  There is a spectral sequence $E_r(\proj_n)$ associated
with this filtration.

\begin{lemma}\label{iso}
  The map $\psi$ induces an isomorphism of complexes $\aff_{n-1} \cong
  E_0(\proj_n)$.
\end{lemma}

\begin{proof}
  Let us filter $\aff_{n-1}=\kont_{n-1} \otimes
  \Lambda(\eta_1,\dots,\eta_{n-1})$ so that $F_k(\aff_{n-1})$ is spanned
  by elements $\gamma \otimes \omega$ where $\omega$ is a product of
  at least $k$ distinct generators $\eta_i$.  The linear isomorphism
  $\psi\colon \aff_{n-1} \to \proj_n$ respects the filtrations.  We show
  that if $\gamma\otimes \omega \in F_k(\aff_{n-1})$ then
  \[
  d\psi(\gamma\otimes \omega) =\psi(d\gamma \otimes
  \omega)+\epsilon(\gamma\otimes \omega),
  \]
  where the error term, $\epsilon(\gamma\otimes \omega)$, is in
  strictly higher filtration, i.e. $\epsilon(\gamma\otimes \omega) \in
  F_{k+1}(\proj_n)$.  This will prove the lemma.  Recall that
  $\psi(\gamma\otimes \omega)$ has an edge from the external vertex
  $u_0$ to each internal vertex $v$; let $h_v$ denote this edge.  The
  differential of $\psi(\gamma \otimes \omega)$ is a sum of terms of
  the form $\psi(\gamma \otimes \omega) \oslash (e,f)$ --- we now
  consider the various possible positions of the pair of adjacent
  edges $(e,f)$ in $\psi(\gamma \otimes \omega)$.

  Case 1: $e$ is an internal edge from $v_1$ to $v_2$.  If $f$
  connects $v_2$ to the external vertex $u_0$, then
  $\psi(\gamma\otimes \omega) \oslash (e,f)= \psi(\gamma/e \otimes
  \omega)$.  To see that the orientations agree, suppose $\gamma$
  has affine orientation $x$. Then $\psi(\gamma\otimes \omega)$ has
  projective orientation
  \[
  x \wedge (\omega) \wedge (h_v \wedge v)_{v\in V_{int}}
 = x \wedge (\omega) \wedge f\wedge v_2 \wedge (h_v \wedge v)_{v\neq v_2}
  \]
  where $(\omega)$ here stands for a wedge product of the boundary
  edges incident at $u_0$ corresponding to the class $\omega$, and
  $\psi(\gamma \otimes \omega) \oslash (e,f)$ has projective
  orientation
  \[
\iota_e x \wedge (\omega) \wedge (h_v \wedge v)_{v\neq v_2},
 \]
 which is precisely the projective orientation of $\psi(\gamma/e
 \otimes \omega)$.  If $f$ connects $v_2$ to a vertex that is not
 $u_0$ then $\psi(\gamma) \oslash (e,f) =0$ because this graph has a
 double edge coming from the edges $u_0 v_1$ and $u_0 v_2$ in
 $\psi(\gamma)$.

  Case 2: $e$ goes from an external vertex $x \neq u_0$ to an internal
  vertex $v$. If $f=h_v$ is the edge from $u_0$ to $v$, then
  $\psi(\gamma\otimes \omega) \oslash (e,f) = \psi(\gamma/e \otimes
  \omega)$ as in case 1.  The remaining possibilities for the position
  of $f$ will contribute only to the error term
  $\epsilon(\gamma\otimes \omega)$.  If $f \neq h_v$ then
  $\psi(\gamma\otimes \omega)\oslash (e,f)$ is in strictly higher
  filtration because it has an additional external edge between $x$
  and $u_0$.

  Case 3: $e$ goes from $u_0$ to an internal vertex $v$.  The internal
  vertices are at least 4-valent, so there are at least 3 edges
  connecting $v$ to vertices $a,b,c$ other than $u_0$.  We may assume
  that $f$ is not the edge from $v$ to $a$.  If $a$ is internal then
  $\psi(\gamma\otimes \omega) \oslash (e,f)=0$ because it has a double
  edge coming from the edges from $a$ to $v$ and from $a$ to $u_0$ in
  $\psi(\gamma\otimes \omega)$.  If $a$ is external then
  $\psi(\gamma\otimes \omega) \oslash (e,f)$ has one more edge than
  $\psi(\gamma\otimes \omega)$ connecting $a$ to $u_0$ and thus lives
  in higher filtration.
\end{proof}

\begin{proof}[Proof of Theorem \ref{projective-graphs-resolution}]
  By Lemma \ref{iso} the $E_1$-term is $E_1(\proj_n) \cong
  H^*(\aff_{n-1})$.  On the other hand by definition of the affine
  graph complex (Definition \ref{agc}), and since Kontsevich's graph
  complex $\kont_{n-1}$ is a resolution of $H^*(D_2(n-1))$,
  \begin{align*}
    H^*(\aff_{n-1}) & \cong 
                    H^*(\kont_{n-1}) \otimes \Lambda(\eta_1,\dots,\eta_{n-1}) \\
                 & \cong
                    H^*(D_2(n-1)) \otimes \Lambda(\eta_1,\dots,\eta_{n-1}) \\
                 & \cong H^*(\M_n).
  \end{align*}
  We claim that the spectral sequence collapses at the $E_1$-term.
  The projection $\proj_n \to H^*(\M_n)$ is surjective in cohomology
  because the forms $\alpha_{ij}$ generate the cohomology ring of
  $\M_n$ and $\alpha_{ij}$ is the image of the graph with a single
  edge between external vertices labelled $i$ and $j$. Thus the dimension
  of $H^*(\proj_n)$ is not smaller than the dimension $H^*(\M_n)$. The
  isomorphism $E_1(\proj_n) \cong H^*(\M_n)$ and finite dimensionality
  imply that $E_1(\proj_n) \cong E_\infty(\proj_n)$ and $q^*\colon H^*(\proj_n)
  \stackrel{\cong}{\to} H^*(\M_n)$ is an isomorphism of cyclic
  cooperads.
\end{proof}

\section{The Kontsevich Integral from graphs to forms}\label{int-section}

In this section we construct a quasi-isomorphism
\[
\I\colon \proj_V \to \Omega_{PA}^*(\M_V),
\]
where $\Omega_{PA}^*$ is the complex of ``piecewise semi-algebraic
forms''.  The idea is as follows.  A graph $\gamma=(V, V_{int},E)$
determines a differential form on $\M_{V\sqcup V_{int}}$ given by
multiplying $\alpha_{ij}$ 1-forms corresponding to the edges, and one
then integrates this form along the fibres of the bundle $\M_{V \sqcup
  V_{int}} \to \M_{V}$ to obtain a form on $\M_V$.  Both the integrand
and the integration map have a sign ambiguity, but if $\gamma$ has a
projective orientation then these sign ambiguities are resolved.  The
differential form that results is generally \emph{not} smooth, and so
we must pass to the semi-algebraic setting in which the Kontsevich
integral is well-defined.

\subsection{PA forms and fibrewise integration}
Given a semi-algebraic manifold, the complex of \emph{minimal} forms,
$\Omega^*_{min}(X)$, is spanned by elements
\[
f_0 \cdot df_1 \wedge \cdots \wedge df_k,
\]
where the $f_i$ are semi-algebraic functions.  The complex of $PA$
forms, $\Omega_{PA}^*(X)$, consists of those forms which can be
written as a pushforward of a minimal form.  When $X$ is smooth, the
commutative differential graded algebra of $PA$ forms is
quasi-isomorphic to the ordinary smooth forms.

We first recall some facts about semi-algebraic fibrewise
integration.  Suppose $\pi\colon E \to B$ is a semi-algebraic fibre
bundle with $n$-dimensional compact fibres and equipped with a fibrewise
orientation.  Integration of forms fibrewise defines a map (see
\cite[Section 8]{HLTV}) 
\[
\pi_!\colon \Omega^*_{min}(E) \to \Omega^{*-n}_{PA}(B).
\]
In particular, if $\omega \in \Omega^n_{min}(E)$ restricts to a volume
form on each fibre then $\pi_! \omega$ is the semi-algebraic function
on $B$ which sends $b \in B$ to the volume of the fibre over $b$.

We shall need \emph{iterated} fibrewise integrals at several points in
this section.  As the theory is developed in \cite{HLTV}, the
pushforward of a minimal form is a $PA$ form by definition, and the
pushforward of a $PA$ form is not defined.  Note that the minimal
forms constitute a sub-CDGA of all $PA$ forms.
\begin{proposition} \label{twosteps}
If $p\colon E \to B$ and
$q\colon B \to X$ are semi-algebraic bundles with compact oriented
fibres, and $\xi$ is a minimal form on $E$ such that $p_! \xi$ is a
minimal form on $B$, then $q_! (p_! \xi)$ is defined and 
\[
q_! (p_! \xi) = (q\circ p)_! \xi
\] 
\end{proposition}
This is a slight generalisation of \cite[Prop 8.11]{HLTV} and has 
a similar proof.

Given any form $\xi \in \Omega^*_{min}(B)$, one has the push-pull
formula (\cite[Prop. 8.13]{HLTV}),
\begin{equation}\label{pushpull}
\pi_! (\pi^* \xi \cdot \omega) = \xi \cdot  \pi_! \omega. 
\end{equation}
Now recall the fibrewise Stokes formula.
\begin{theorem} \cite[Prop. 8.10]{HLTV} \label{Stokes}
  Let $\pi\colon E \to B$ be a semi-algebraic bundle of oriented
  $k$-manifolds with boundary, with $\pi|_{\partial}
  \colon \partial E \to B$ denoting the associated 
  fibrewise boundary bundle.  Then, giving the
  boundary the induced orientation,
\[
d ( \pi_! \xi ) = \pi_! (d \xi) + (-1)^{\mathrm{deg} \xi - k}(\pi|_{\partial})_! \xi,
\]
for any form $\xi \in \Omega^*_{min}(E)$.
\end{theorem}

\subsection{Definition of the integration map $\I$}
We first define $\I$ on the subcomplex $\proj_V^0$
of graphs with no internal vertices.  Let $\gamma \in \proj_V^0$ be
such a graph with edge set $E$ and orientation ray in $\det(E)$
spanned by $e_1 \wedge \cdots \wedge e_k$.  An edge $e$ between
vertices $u$ and $v$ determines a semi-algebraic minimal 1-form
\[
\alpha_e \coloneqq \alpha_{u v} = \pi_{u v}^* d\theta \in
\Omega^1_{min}(\M_V),
\]
where $\pi_{u v}\colon \M_{V} \to \M_{\{u,v\}} \cong S^1$ forgets all
marked points except $u$ and $v$.  Define
\[
\I(\gamma) \coloneqq \alpha_{e_1} \wedge \cdots \wedge \alpha_{e_k}
\in \Omega^*_{min}(\M_V).
\]

Recall that (the equivalence class of) a projectively oriented graph $\gamma$ with internal vertex
set $V_{int}$ and external vertex set $V$ can be written as $\pi^{V_{int}}_r(\widetilde{\gamma}) \in \proj_V$
for some ray $r \subset \det(V_{int})$ and graph
$\widetilde{\gamma}\in \proj_{V\sqcup V_{int}}^0$.  
We then define
\[
\I(\gamma) = \I(\pi^{V_{int}}_r (\widetilde{\gamma})) \coloneqq
\frac{1}{2^{|V_{int}|}} \pi^{V_{int}}_! \I(\widetilde{\gamma}),
\]
where the fibrewise orientation of $\pi^{V_{int}}\colon \M_{V\sqcup
  V_{int}} \to \M_V$ is determined by the ray $r$ as described in
section \ref{orientations}.  We prove in Lemma
\ref{Kont-integral-properties} below that $\I$ respects the relations
$(R1-4)$ of Definition \ref{defproj}, and hence it descends to a map
from $\proj_V$, that we denote by the same name.

\begin{lemma}\label{Kont-integral-properties}
The Kontsevich Integral satisfies
\begin{itemize}
\item[(i)] $\I(\gamma) = - \I\left(\overline{\gamma}\right)$.
\item[(ii)] $\I$ vanishes on graphs having an internal vertex of
  valence $\leq 3$.
\item[(iii)] $\I$ vanishes on graphs which have a connected component
  containing at most one vertex in $V$ and at least
  one internal vertex.
\item[(iv)] The pinwheel ideal is contained in the kernel of $\I$.
\end{itemize}
Hence we have a well-defined morphism $\I\colon \proj_V \to \Omega_{PA}^*(\M_V)$
of graded algebras.
\end{lemma}
\begin{proof}
  Statement (i) is trivial.  For (ii), suppose $\gamma$ has a trivalent internal
  vertex $x$ adjacent to vertices $u$, $v$ and $w$ as below.
  \begin{center}
  \input{trivalent.pst}  
  \end{center}
  By Proposition \ref{twosteps} one may integrate out the internal
  vertices by first integrating out $x$ and then integrating out the
  remaining internal vertices, i.e.
  \[
  \pi^{V_{int}}_! \I(\widetilde{\gamma}) = (-1)^{|V_{int}|-1}\pi^{V_{int} \smallsetminus \{x\}}_!
  \pi^{\{x\}}_! \I(\widetilde{\gamma}),
  \]
  where if $\pi^{V_{int}}$ has fibrewise orientation $r\subset
  \det(V_{int})$ then $\pi^{V_{int} \smallsetminus \{x\}}$ has
  fibrewise orientation $\iota_x r \subset \det(V_{int} \smallsetminus
  \{x\})$ and $\pi^{\{x\}}$ has its standard fibrewise orientation.
  This works because the integrand $\I(\widetilde{\gamma})$ is the
  product of $\alpha_{ux} \alpha_{vx} \alpha_{wx}$ with a minimal form
  $\beta$ that is independent of $x$, so by formula \eqref{pushpull},
  $(\pi^{\{x\}})_!\I(\widetilde{\gamma})$ is equal to the product of
  $\beta$ with the $PA$ form
  \begin{equation}\label{trivalent-integral}
  \pi^{\{x\}}_! (\alpha_{ux} \alpha_{vx} \alpha_{wx}), 
  \end{equation}
  and since the fibres of $\pi^{\{x\}}\colon \M_{V \sqcup V_{int}} \to
  \M_{V \sqcup V_{int} \smallsetminus \{x\}}$ have dimension 3,
  \eqref{trivalent-integral} is a semi-algebraic function on
  $\M_{V \sqcup V_{int} \smallsetminus \{x\}}$, and so
  $(\pi^{\{x\}})_!\I(\widetilde{\gamma})$ is minimal.  Putting
  $u$, $v$ and $w$ at $0$, $1$ and $\infty$ identifies each fibre with
  the complement of a codimension 1 semi-algebraic subset of the unit
  tangent circle bundle of the Riemann sphere punctured at $0$, $1$
  and $\infty$, $S(T(S^2 \smallsetminus \{0,1,\infty\}))$, and the
  differential form $\alpha_{ux} \alpha_{vx} \alpha_{wx}$ extends to
  $S(T(S^2 \smallsetminus \{0,1,\infty\}))$.  Complex conjugation is
  an orientation-preserving diffeomorphism of the tangent circle
  bundle, and it reverses the sign of $\alpha_{ux} \alpha_{vx}
  \alpha_{wx}$ since it reverses the sign of each of the three
  factors.  Hence the integral over each fibre is zero, and the
  fibrewise integral \eqref{trivalent-integral} is identically zero.

  Statement (iii) follows from Kontsevich's lemma on vanishing of integrals over
  configuration spaces \cite[Lemma 6.4]{Ko1}. We choose an external
  vertex as the point at infinity, and write a form
  $\I(\widetilde{\gamma})$ as polynomial in the forms
  $\omega_{ij}$ and $\eta_k$.  Integrate out the internal vertices by
  first integrating out the tangent ray decorations at the internal vertices,
  resulting in a $PA$ and in fact minimal form on a configuration
  space; what remains is a configuration space integral of the form
  that Kontsevich's lemma refers to.

  For (iv) We will in show that if
  $v$ is an internal vertex then integrating out the tangent ray
  decoration at $v$ sends the element
  \[
  \sum_{e\in E(v)} \I(\widetilde{\gamma} \smallsetminus e)
  \]
  to zero.  Write the orientation of $\widetilde{\gamma}$ as $\beta
  \wedge e_0 \wedge \cdots \wedge e_k$ where $E(v) = \{e_0, \ldots,
  e_k\}$ is the set of edges incident at $v$.  Let $u_k$ denote the opposite
  endpoint of $e_k$.  Then
  \begin{align*}
    \sum_{e\in E(v)} & \widetilde{\gamma} \smallsetminus e = 
    \beta \sum_{i=0}^k (-1)^{i} \alpha_{e_0} \alpha_{e_1} \cdots
    \widehat{\alpha_{e_i}} \cdots \alpha_{e_k} \\
    = & \beta(\eta_v + \eta_{u_1} -2\omega_{v u_1})
          \cdots (\eta_v + \eta_{u_k} -2\omega_{v u_k}) \\
    & +  \beta \eta_v \sum_{i=1}^k (-1)^{i} (\eta_v + \eta_{u_1}
    -2\omega_{v u_1}) \cdots \widehat{(\eta_v + \eta_{u_i} -2\omega_{v
        u_i})}
    \cdots (\eta_v + \eta_{u_k} -2\omega_{v u_k}),
  \end{align*}
  and integrating out the decoration at $v$ extracts the coefficient of
  $\eta_v$.  It is not hard to see that the coefficient of $\eta_v$ is
  in fact zero.
\end{proof}

\subsection{Differentials of graph forms}
Here we shall prove the following.
\begin{theorem}\label{differential-of-graph}
  The linear map $\I\colon \proj_V \to \Omega_{PA}^*(\M_V)$ commutes with the
  differentials and is thus a morphism of differential graded algebras
\end{theorem}

Let $\gamma \in \proj_V$ be a projectively oriented graph.  We will
prove this theorem by using the Stokes formula to compute the exterior
derivative of the PA form $\I(\gamma)$.  Write $\gamma =
\pi_r^{V_{int}}(\widetilde{\gamma})$.  By the fibrewise Stokes formula
(Theorem \ref{Stokes}),
\begin{equation}\label{stokes-graph-forms}
  d\I(\gamma) = d\left( \frac{1}{2^{|V_{int}|}}\pi^{V_{int}}_!
    \I(\widetilde{\gamma}) \right) =
  \frac{(-1)^{|E|+|V_{int}|}}{2^{|V_{int}|}}
  \sum_{U_1,U_2} \left( \pi^{V_{int}}|_{\partial_{U_1,U_2}} \right)_! \I(\widetilde{\gamma}),
\end{equation}
since $\I(\widetilde{\gamma})$ is closed.  We can factor
$\pi^{V_{int}}|_{\partial_{U_1,U_2}}$ into the composition of first
forgetting the points in $U_1 \cap V_{int}$ and then forgetting the
remaining internal vertices:
\[
\partial_{U_1,U_2} \xrightarrow{\pi^{U_1 \cap
    V_{int}}|_{\partial_{U_1,U_2}}} \M_{V \cup U_2}
\xrightarrow{\pi^{U_2 \cap V_{int}}} \M_V.
\] 
By Lemma \ref{fibrewise-boundary}, the fibrewise boundary component
$\partial_{U_1,U_2}$ is the (closure of the) locus where a node
separates the sphere into two lobes $S_1, S_2$, with the points
labelled by $U_i$ on the lobe $S_i$.  The map $\pi^{U_1 \cap
  V_{int}}|_{\partial_{U_1,U_2}}$ forgets the first lobe if $U_1 \cap
V = \emptyset$, and collapses it to a single marked point if $U_1$
contains an external vertex; hence it is a semi-algebraic fibre
bundle, and in fact it is trivial.

\begin{lemma}
  If $|U_1| \geq 3$ then the corresponding term in the Stokes formula vanishes:
  \[
  \left(\pi^{V_{int}}|_{\partial_{U_1,U_2}}\right)_!
  \I(\widetilde{\gamma}) = 0.
  \]
\end{lemma}
\begin{proof}
  Consider performing the fibrewise integration by first integrating along
  the fibres of $\pi^{U_1 \cap V_{int}}|_{\partial_{U_1,U_2}}$ and then integrating out the
  remaining internal vertices.  Furthermore, consider performing the
  integral over the fibres of $\pi^{U_1 \cap
    V_{int}}|_{\partial_{U_1,U_2}}$ by first integrating
  out the tangent ray decorations of the points in $U_1 \cap V_{int}$ and then
  integrating out the positions of these points:
  \[
  (\pi^{U_1 \cap
    V_{int}}|_{\partial_{U_1,U_2}})_! = (\pi_{\mathrm{pos}})_! \circ
  (\pi_{\mathrm{dec}})_!
  \]
  Since iterated integration of $PA$ forms is not defined in general,
  we must justify as above both of these iterated integrations.  We will see
  that each successive integration results again in a minimal form, so
  the iteration is allowed.

  Relative to a choice of a point $x \in U_2$ as point at infinity,
  the form $\I(\widetilde{\gamma})$ can be expanded in terms of the
  affine presentation generators as in the proof of Lemma
  \ref{Kont-integral-properties} (iii),(iv).  Integrating out the
  $V_{int} \cap U_1$ decorations then yields a polynomial $P =
  (\pi_{\mathrm{dec}})_!\I(\widetilde{\gamma})$ in $\omega_{ij}$ forms
  and $\eta_k$ forms for $k\notin V_{int} \cap U_1$, and hence a
  minimal form.  On the locus $\partial_{U_1,U_2}$, let $q$ denote the
  node between the two lobes.  On this locus, if $a\in U_1$ and $b\in
  U_2$ then $\omega_{ab} = \omega_{qb}$.  Hence each monomial in $P$
  factors as a product $\beta_2 \wedge \beta_1$, where $\beta_2$ is a
  monomial of $\omega$ and $\eta$ forms on the second lobe (relative to $x$ and
  including $q$ as a marked point) and $\beta_1$ is a monomial of
  only $\omega$ forms on the first lobe (relative to $q$).  Now,
  integrating out the $V_{int} \cap U_1$ positions,
\[
(\pi_{\mathrm{pos}})_! (\beta_2 \beta_1) = \beta_2 \cdot
(\pi_{\mathrm{pos}})_!(\beta_1)
\]
and if $|U_1| \geq 3$ then $(\pi_{\mathrm{pos}})_!(\beta_1) = 0$ by
Kontsevich's Vanishing Lemma \cite[Lemma 6.4]{Ko1}.
\end{proof}

Thus the only fibrewise boundary components which contribute are those
for which $|U_1| = 2$.  If $U_1=\{u,v\}$ then we will write
$\partial_{uv} = \partial_{U_1,U_2}$.  We will now compute the Stokes
boundary term
$(\pi^{V_{int}}|_{\partial_{uv}})_!\I(\widetilde{\gamma})$.
Factor $\pi^{V_{int}}|_{\partial_{uv}}$ as 
\[
  \partial_{uv}
  \xrightarrow{\pi^{\{v\}}|_{\partial_{uv}}} 
  \M_{V \sqcup V_{int} \smallsetminus \{v\}}
  \xrightarrow{\pi^{V_{int} \smallsetminus \{v\}}}
  \M_{V}.
\]
The first map is given its standard fibrewise orientation and 
the second map is given the fibrewise orientation $\iota_x r$, where 
$r \subset \det(V_{int})$ is the fibrewise orientation of
$\pi^{V_{int}}$.  Note that with these choices of orientation,
\[
\pi^{V_{int}\smallsetminus \{v\}}_! \circ \pi^{\{v\}}_! =
(-1)^{|V_{int}|-1}\pi^{V_{int}}_!.
\]
The Stokes formula \eqref{stokes-graph-forms} then becomes
\begin{equation}\label{stokes-graph-forms-2}
  d\I(\gamma) = \frac{(-1)^{|E|-1}}{2^{|V_{int}|}}\sum_{\{u,v\}} \pi^{V_{int} \smallsetminus \{v\}}_!
  (\pi^{\{v\}}|_{\partial_{uv}})_! \I(\widetilde{\gamma}),
\end{equation}
where the sum runs over unordered pairs $\{u,v\} \subset V \sqcup
V_{int}$ at most one of which is external (if one is internal then we
call it $v$, and if both are internal then we choose one to call $v$).

\begin{lemma}\label{key-relation}
Suppose we are given four distinct vertices $u, v, w, x$.  Then on
$\partial_{u v}$,
\[
\alpha_{u w} \alpha_{u x} = (1/2) (\alpha_{v w} - \alpha_{v x})(\alpha_{u w} + \alpha_{u x}).
\]
\end{lemma}
\begin{proof}
  We can consider $x$ to be at infinity and in the corresponding
  affine presentation, $\omega_{u w} = \omega_{v w}$.  Hence, on
  $\partial_{u v}$ we have the relation $\alpha_{u x} - \alpha_{u w} =
  \alpha_{v x} - \alpha_{v w}$.  The result follows from multiplying
  both sides with $\alpha_{vx} + \alpha_{vw}$ since 1-forms
  anticommute.
\end{proof}

\begin{lemma}\label{volume-forms}
  For any triple of vertices $(u,v,w)$, the 2-form 
  \[
  (1/2)\alpha_{w v} \alpha_{v u}
  \]
  is a fibrewise negative unit volume form for
  $\pi^{\{v\}}|_{\partial_{uv}} \colon \partial_{u v} \to \M_{V \sqcup
    V_{int} \smallsetminus \{v\}}$.
\end{lemma}
\begin{proof}
  In terms of the affine presentation (relative to $w$), one sees that
  $\omega_{u v} \eta_{v}$ is a fibrewise negative unit volume form for
  $\pi^{\{v\}}|_{\partial_{uv}}\colon \partial_{u v} \to \M_{V\sqcup
    V_{int} \smallsetminus \{v\}}$ as follows.  We factor
  $\pi^{\{v\}}$ as $\pi^{\{v\}}_{\mathrm{pos}} \circ
  \pi^{\{v\}}_{\mathrm{dec}}$, and each of these maps has a fibrewise
  orientation as described in section \ref{orientations}.  The form
  $\eta_v$ is a fibrewise positive unit volume form for
  $\pi^{\{v\}}_{\mathrm{dec}}$.  The form $\omega_{uv}$ is a negative
  fibrewise unit volume form for the restriction of
  $\pi^{\{v\}}_{\mathrm{pos}}$ to the face
  $\pi^{\{v\}}_{\mathrm{dec}}(\partial_{uv}) = \partial_{uv}/S^1
  \subset \M_{V \sqcup V_{int}}/S^1$. Translating to the cyclic
  presentation,
  \begin{align*}
    \omega_{u v} \eta_{v}
    & = (1/2)(\alpha_{w u} + \alpha_{w v} - \alpha_{u v})\alpha_{w v} \\
    & = (1/2)(\alpha_{w u} \alpha_{w v} - \alpha_{u v} \alpha_{w v}).
  \end{align*}
  Since $\alpha_{w u}$ is constant on the fibres of
  $\pi^{\{v\}}|_{\partial_{uv}}$, the first term in the final line
  integrates to zero on each fibre, and so we obtain the result.
\end{proof}

\begin{lemma}\label{boundary-expansion}
Consider a vertex set $W=\{1, \ldots, k, u,v\}$.  On the boundary
component $\partial_{uv} \subset \M_W$ the following identity holds.
\[
\alpha_{1v} \cdots \alpha_{kv} =
 (1/2)^{k-1} \left( \sum_{i=1}^k (-1)^{k+i} 
                    \alpha_{1u} \cdots \widehat{\alpha_{iu}}
                               \cdots \alpha_{ku} \right)
             \left( \alpha_{1v} + 
                    \sum_{j=2}^k 2^{j-2}\alpha_{jv}  \right)
\]
\end{lemma}
\begin{proof}
Use induction on $k$ and the relation \eqref{key-relation}.
\end{proof}

\begin{proof}[Proof of Theorem \ref{differential-of-graph}]
  Suppose $\gamma$ is a projective graph.  Write $\gamma =
  \pi^{V_{int}}_r (\widetilde{\gamma})$.  Given an edge $e$ between
  vertices $u,v$ with $v$ internal (if $u$ is internal as well then we
  orient $e$ with head at $v$) and a second edge $f$ meeting
  $v$,
\[
\gamma \oslash (e,f) = (-1)^{|E|} \pi^{V_{int} \smallsetminus \{v\}}_{\iota_v r} (\widetilde{\gamma}
\oslash (e,f)),
\]
where the sign comes from sliding $v$ to the left past the edges $E
\smallsetminus \{f,e\}$.  By Lemmas \ref{boundary-expansion} and
\ref{volume-forms},
\[
(\pi^{\{v\}}|_{\partial_{uv}})_! \I(\widetilde{\gamma}) = 
  (-2) \sum_{f \in E(v) \smallsetminus \{e\}}
    \I(\widetilde{\gamma} \oslash (e,f)).
\]
By combining these we find that
\begin{align*}
\sum_f \I(\gamma\oslash (e,f)) = 
   & \left(\frac{(-1)^{|E|-1}}{2^{|V_{int}|-1}}\right) \sum_f 
     \pi^{V_{int} \smallsetminus \{v\}}_! \I(\widetilde{\gamma} \oslash (e,f)) \\ 
 = & \left(\frac{(-1)^{|E|-1}}{2^{|V_{int}|}} \right) \pi^{V_{int} \smallsetminus \{v\}}_!
     (\pi^{\{v\}}|_{\partial_{uv}})_! \I(\widetilde{\gamma}). 
\end{align*}
Summing this over all unordered pairs $\{u,v\}$ and using
\eqref{stokes-graph-forms-2} then gives the desired result.
\end{proof}

\subsection{Example: resolving the cyclic Arnold relation}
As an illustration of Theorem \ref{differential-of-graph}, let us
consider the product $(1/2)\alpha_{v a} \alpha_{v b} \alpha_{v c}
\alpha_{v d}$ corresponding to the graph $\gamma$ shown below,
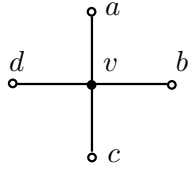
\begin{center}
\input{plus-graph.pst}
\end{center}
with internal vertex $v$ and projective orientation $(e_{av} \wedge
e_{bv} \wedge e_{cv} \wedge e_{dv} \wedge v)$.  We can express
$(\pi^{\{v\}}|_{\partial_{dv}})_!  \I(\widetilde{\gamma})$ as a
linear combination of edge forms as follows.  On $\partial_{dv}$,
\begin{align*}
\alpha_{v a} \alpha_{v b} \alpha_{v c} \alpha_{d v}
& = (1/2)(\alpha_{d a} - \alpha_{d b})
             (\alpha_{v a} + \alpha_{v b})\alpha_{v c} \alpha_{d v}\\
& = (1/4)(\alpha_{d a} - \alpha_{d b})(\alpha_{d a} - \alpha_{d c})
             (\alpha_{v a} + \alpha_{v c}) \alpha_{d v}\\
& \:\:\:\:   + (1/4)
       (\alpha_{d a} - \alpha_{d b})
       (\alpha_{d b} - \alpha_{d c})
       (\alpha_{v b} + \alpha_{v c})\alpha_{d v} \\
& = (1/4)(\alpha_{da} \alpha_{db} - \alpha_{da} \alpha_{dc} + \alpha_{db} \alpha_{dc})
              (\alpha_{va}\alpha_{dv} + \alpha_{vb}\alpha_{dv} + 2\alpha_{vc}\alpha_{dv})
\end{align*}
By Lemma \ref{volume-forms}, 
\[
(\pi^{\{v\}}|_{\partial_{dv}})_! ((1/2) \alpha_{v a} \alpha_{v b} \alpha_{v c} \alpha_{v d}) = 
-(\alpha_{da} \alpha_{db} - \alpha_{da} \alpha_{dc} + \alpha_{db} \alpha_{dc}).
\]
This corresponds to the sum of graphs
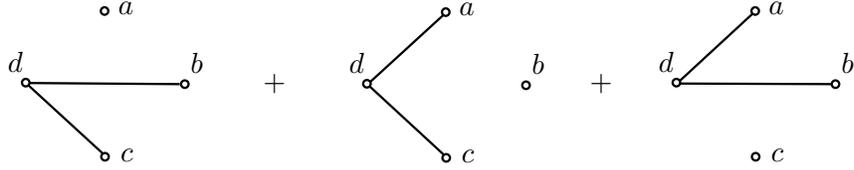
\begin{center}
\input{plus-graph-differential.pst}
\end{center}
where each is given the orientation induced from that of $\gamma$ by
first deleting an edge and then contracting the $\{d,v\}$ edge.  From
this one finds that 
\[
d(\mbox{plus graph}) = (\mbox{cyclic Arnold relation}).
\]

\section{Compatibility of the Kontsevich integral with the cyclic
  operad structure}

The $PA$-forms on $\M$ do not constitute a cyclic cooperad because the
exterior product of forms, $\Omega_{PA}^*(X) \otimes \Omega_{PA}^*(Y)
\to \Omega_{PA}^*(X\times Y)$ is invertible up to quasi-isomorphism
but not on the nose.  However, Lemma \ref{cooperad-diagram} below
shows compatibility between the cyclic operad structure on $\M$ and
the cyclic cooperad structure on $\proj$.

\begin{lemma}\label{pullback-lemma}
  For $u,v \in (I\smallsetminus\{i\}) \sqcup (J \smallsetminus
  \{j\})$, the pullback of $\alpha_{uv}$ via $_i\circ_j\colon \M_{I} \times
  \M_{J} \to \M_{I\sqcup J \smallsetminus \{i,j\}}$ is given by:
\[
{(_i\circ_j)^*} \alpha_{uv} = 
\begin{cases} 
\alpha_{uv} \otimes 1 & \mbox{if $u,v \in I$,} \\
1\otimes \alpha_{uv} & \mbox{if $u,v \in J$,} \\
\alpha_{ui}\otimes 1 + 1 \otimes \alpha_{jv} & \mbox{if $u\in I$ and $v \in J$.}
\end{cases}
\]
\end{lemma}
\begin{proof}
  If $u$ and $v$ are both in $I$ then commutativity of the diagram
  \[
  \begin{diagram}
    \node{\M_I \times \M_J}
    \arrow{e,t}{_i\circ_j} \arrow{s}
    \node{\M_{I\sqcup J\smallsetminus \{i,j\}}} \arrow{s,r}{\pi_{\{u,v\}}} \\
    \node{\M_I \times \{pt\}} \arrow{e,t}{\pi_{\{u,v\}}} \node{\M_{\{u,v\}}}
  \end{diagram}
  \]
  implies the formula in this case; the reasoning is the same when
  $u,v \in J$.  Suppose now that $u\in I \smallsetminus \{i\}$ and $v
  \in J\smallsetminus \{j\}$.
  Choosing an identification $\M_2 \cong S^1$, the map
  \[
  _i\circ_j \colon \M_{\{u,i\}}\times \M_{\{j,v\}} \to \M_{\{u,v\}}
  \]
  corresponds with the product map $\mu\colon S^1 \times S^1 \to S^1$, and
  $\mu^* d\theta = d\theta \otimes 1 + 1\otimes
  d\theta$.  Commutativity of the diagram
  \[
  \begin{diagram}
    \node{\M_{I} \times \M_{J}}
    \arrow{e,t}{_i\circ_j} \arrow{s,l}{\pi_{\{u,i\}} \times \pi_{\{j,v\}}}
    \node{\M_{I\sqcup J}} \arrow{s,r}{\pi_{\{u,v\}}} \\
    \node{\M_{\{u,i\}} \times \M_{\{j,v\}}} \arrow{e,t}{_i\circ_j}
    \node{\M_{\{u,v\}}}
  \end{diagram}
  \]
  then implies the formula.
\end{proof}

\begin{lemma}\label{cooperad-diagram}
  (c.f. \cite[Proposition 8.19]{LambVol}.)  The following diagram
  commutes:
\[
\begin{diagram}
  \node{\proj_{I\sqcup J}} \arrow{s,l}{\I} \arrow[2]{e,t}{_I\bullet_J} 
\node[2]{\proj_{I \sqcup \{x\}} \otimes \proj_{J\sqcup\{y\}}}
  \arrow{s,r}{\I\otimes \I} \\
  \node{ \Omega_{PA}^*(\M_{I\sqcup J})}
  \arrow{se,t}{({_x\circ_y})^*}
  \node[2]{\Omega_{PA}^*(\M_{I\sqcup\{x\}}) \otimes
    \Omega_{PA}^*(\M_{J\sqcup\{y\}}).} \arrow{sw} \\
  \node[2]{\Omega_{PA}^*(\M_{I\sqcup \{x\}} \times \M_{J\sqcup \{y\}})}
\end{diagram}
\]
\end{lemma}
\begin{proof}
  The composition map $_x\circ_y\colon \M_{I\sqcup \{x\}} \times
  \M_{J\sqcup \{y\}} \to \M_{I\sqcup J}$ is a principal $S^1$-bundle
  over its image $\partial_{I,J}$, where the circle action is by
  rotating the rays at the $x$ and $y$ points in opposite directions.
  Consider a graph $\gamma \in \proj_{I\sqcup J}$ with set of internal
  vertices $V_{int}$.  There are pullback squares
    \[
    \begin{diagram}
      \node{\coprod_{V_{int} = A\sqcup B} \M_{I\sqcup A\sqcup \{x\}}
        \times \M_{J\sqcup B\sqcup \{y\}}} \arrow{e,t}{\coprod {_x\circ_y}}
      \arrow{s,r}{\coprod \pi^A \times \pi^B}
      \node{\coprod_{V_{int} = A\sqcup B} \partial_{I\sqcup A, J\sqcup
          B}}
      \arrow{s} \\
      \node{\M_{I\sqcup\{x\}} \times \M_{J\sqcup \{y\}}} \arrow{e,t}{_x\circ_y}
      \node{\partial_{I,J}}
    \end{diagram}
    \]
    and
    \[
    \begin{diagram}
      \node{\bigcup_{V_{int} = A \sqcup B} \partial_{I\sqcup A,J\sqcup
          B}} \arrow{e,J} \arrow{s} \node{\M_{I\sqcup J \sqcup
          V_{int}}}
      \arrow{s,r}{\pi^{V_{int}}} \\
      \node{\partial_{I,J}} \arrow{e,J} \node{\M_{I\sqcup J}.}
\end{diagram}
\]
On each fibre over $\partial_{I,J}$ the quotient map
$\coprod \partial_{I\sqcup A,J\sqcup B} \to \bigcup \partial_{I\sqcup
  A,J\sqcup B}$ is a diffeomorphism in the interior
   and finite--to--1 on
the boundary.  With this observation and the above pullback squares,
one has 
\begin{align*}
  {(_x\circ_y)^*} \I (\gamma) & = {(_x\circ_y)^*} \left(
    \frac{1}{2^{|V_{int}|}}
    \left(\pi^{V_{int}}\right)_!  \I(\widetilde{\gamma})\right) \\
  & =  \frac{1}{2^{|A|+|B|}} \sum_{A \sqcup
    B=V_{int}} \left(\pi^A\right)_! \otimes \left(\pi^B\right)_!
  \left({({_x\circ_y})^*} \I(\widetilde{\gamma} )\right),
\end{align*}
where we choose compatible fibrewise orientations for $\pi^A$ and
$\pi^B$.  The lemma now follows directly from this, Lemma
\ref{pullback-lemma}, and the definition of the cooperad structure on
$\proj$ from section \ref{cooperad-structure}.
\end{proof}

\subsection{Proof of Theorem \ref{main}}

The projective graph complexes $\{\proj_V\}$ collectively form a
cyclic cooperad in the category of commutative DGAs, and they are
degree-wise finite dimensional.  Hence the graded duals
$\{\proj^\vee_V\}$ form a cyclic operad in the category of
cocommutative differential graded coalgebras.  The projections
$\proj_V \to H^*(\M_V;\R)$ determine a quasi-isomorphism of cyclic
cooperads, so the dual morphisms,
\[
H_*(\M_V;\R) \to \proj_V^\vee
\] 
constitute a quasi-isomorphism of cyclic operads.  Let \[C^{SA}_*(-)\colon
\mbox{Semi-algebraic sets} \to \mbox{Chain complexes}\] denote the
symmetric monoidal functor of semi-algebraic chains with real
coefficients \cite{HLTV}.  This functor is weakly equivalent to the
functor of real singular chains $C^{sing}_*(-)$, so the cyclic operads
$C^{sing}_*(\M;\R)$ and $C^{SA}_*(\M)$ are weakly equivalent.  By
\cite[Prop. 7.3]{HLTV}, if $X$ is compact then the natural evaluation
map
\[
C^{SA}_*(X) \to \Omega^*_{PA}(X)^\vee 
\]
is a quasi-isomorphism.  Although the complexes
$\Omega_{PA}^*(\M_V)^\vee$ do not form a cyclic operad, the compositions
\[
C^{SA}_*(\M_V) \to \Omega^*_{PA}(\M_V)^\vee \xrightarrow{\I^\vee} \proj_V^\vee  
\]
give a quasi-isomorphism of cyclic operads by Lemma
\ref{cooperad-diagram}.  This completes the proof.


\end{document}